\documentclass[a4paper]{amsart}
\usepackage[foot]{amsaddr}

\usepackage[latin1]{inputenc}
\usepackage{amssymb,amsmath,amsthm}
\usepackage{amsfonts}
\usepackage{amsaddr}
\usepackage{enumitem}
\usepackage{booktabs}
\usepackage{ifthen}
\usepackage{tikz}
\usetikzlibrary{matrix,calc,decorations.markings}
\usepackage{url}
\usepackage{hyphenat}
\usepackage{hyperref}

\theoremstyle{plain}
\newtheorem{theorem}{Theorem}[section]
\newtheorem{proposition}[theorem]{Proposition}
\newtheorem{lemma}[theorem]{Lemma}

\newtheoremstyle{clm}
  {}
  {}
  {\itshape}
  {}
  {\rmfamily}
  {\@.}
  { }
  {}
\theoremstyle{clm}
\newtheorem{clm}{Claim}
\numberwithin{clm}{theorem}
\newcommand{\clmqed}[1]{\renewcommand{\qedsymbol}{$\square_\text{Claim~\ref{#1}}$}}

\theoremstyle{definition}
\newtheorem{definition}[theorem]{Definition}
\newtheorem{example}[theorem]{Example}

\newtheorem{remark}[theorem]{Remark}

\newcommand{\AAA}{\ensuremath{\mathbb{A}}}
\newcommand{\IN}{\ensuremath{\mathbb{N}}}
\newcommand{\ZZ}{\ensuremath{\mathbb{Z}}}

\newcommand{\nset}[1]{\ensuremath{[{#1}]}}
\newcommand{\Nout}[2][]{\ensuremath{N_\mathrm{o}^{#1}({#2})}}
\newcommand{\GA}[1]{\ensuremath{\mathbb{A}({#1})}}
\newcommand{\divides}{\ensuremath{\mid}}
\newcommand{\card}[1]{\ensuremath{\lvert{#1}\rvert}}

\newcommand{\Htt}[1][t,t']{\ensuremath{\ifthenelse{\equal{#1}{}}{H}{H_{#1}}}}
\newcommand{\Mtt}[1][t,t']{\ensuremath{\ifthenelse{\equal{#1}{}}{M}{M_{#1}}}}
\newcommand{\Ltt}[1][t,t']{\ensuremath{\ifthenelse{\equal{#1}{}}{L}{L_{#1}}}}
\newcommand{\Ytt}[1][t,t']{\ensuremath{\ifthenelse{\equal{#1}{}}{Y}{Y_{#1}}}}
\newcommand{\Ztt}[1][t,t']{\ensuremath{\ifthenelse{\equal{#1}{}}{Z}{Z_{#1}}}}
\newcommand{\Deltatt}[1][t,t']{\ensuremath{\ifthenelse{\equal{#1}{}}{\Delta}{\Delta_{#1}}}}
\newcommand{\OMEGAtt}[1][t,t']{\ensuremath{\ifthenelse{\equal{#1}{}}{\Omega}{\Omega_{#1}}}}
\newcommand{\omegatt}[1][t,t']{\ensuremath{\ifthenelse{\equal{#1}{}}{\omega}{\omega_{#1}}}}
\newcommand{\Lambdatt}[1][t,t']{\ensuremath{\ifthenelse{\equal{#1}{}}{\Lambda}{\Lambda_{#1}}}}
\newcommand{\lambdatt}[1][t,t']{\ensuremath{\ifthenelse{\equal{#1}{}}{\lambda}{\lambda_{#1}}}}

\newcommand{\MG}[1][G]{\ensuremath{\ifthenelse{\equal{#1}{}}{M}{M_{#1}}}}
\newcommand{\PG}[1][G]{\ensuremath{\ifthenelse{\equal{#1}{}}{P}{P_{#1}}}}
\newcommand{\EG}[1][G]{\ensuremath{\ifthenelse{\equal{#1}{}}{E}{E_{#1}}}}
\newcommand{\OG}[1][G]{\ensuremath{\ifthenelse{\equal{#1}{}}{O}{O_{#1}}}}
\newcommand{\ZG}[1][G]{\ensuremath{\ifthenelse{\equal{#1}{}}{Z}{Z_{#1}}}}
\newcommand{\BG}[1][G]{\ensuremath{\ifthenelse{\equal{#1}{}}{B}{B_{#1}}}}
\newcommand{\omegaG}[1][G]{\ensuremath{\ifthenelse{\equal{#1}{}}{\omega}{\omega_{#1}}}}
\newcommand{\lambdaG}[1][G]{\ensuremath{\ifthenelse{\equal{#1}{}}{\lambda}{\lambda_{#1}}}}

\newcommand{\thn}[2]{\ensuremath{D_{#1}(#2)}}
\newcommand{\loopedge}{\ensuremath{\raisebox{-1.4pt}{\rotatebox{90}{$\circlearrowleft$}}}}

\DeclareMathOperator{\lcm}{lcm}
\DeclareMathOperator{\var}{var}

\begin{document}
\title[Associative spectra of graph algebras I]{Associative spectra of graph algebras I. \\ Foundations, undirected graphs, antiassociative graphs}

\author{Erkko Lehtonen}

\address[E. Lehtonen]%
   {Technische Universit\"at Dresden \\
    Institut f\"ur Algebra \\
    01062 Dresden \\
    Germany
    \and
    Centro de Matem\'atica e Aplica\c{c}\~oes \\
    Faculdade de Ci\^encias e Tecnologia \\
    Universidade Nova de Lisboa \\
    Quinta da Torre \\
    2829-516 Caparica \\
    Portugal}

\author{Tam\'as Waldhauser}

\address[T. Waldhauser]%
   {University of Szeged \\
    Bolyai Institute \\
    Aradi v\'ertan\'uk tere 1 \\
    H-6720 Szeged \\
    Hungary}

\begin{abstract}
Associative spectra of graph algebras are examined with the help of homomorphisms of DFS trees.
Undirected graphs are classified according to the associative spectra of their graph algebras; there are only three distinct possibilities: constant $1$, powers of $2$, and Catalan numbers.
Associative and antiassociative digraphs are described, and associative spectra are determined for certain families of digraphs, such as paths, cycles, and graphs on two vertices.
\end{abstract}

\maketitle


\section{Introduction}

Associativity is a fundamental property of binary operations, and one tends to take
it for granted, since the most frequently encountered operations are associative.
However, there are also many noteworthy operations that are not associative, such as
subtraction, cross product of vectors, implication, just to name a few.
For a systematic study of phenomena related to (non)associativity, one may
consider an arbitrary nonempty set $A$ together with a binary operation
$ x\cdot y$ on $A$. Let us emphasize that we denote the operation as multiplication
only for notational convenience; the operation can be \emph{any} map
$A \times A \to A,\ (x,y) \mapsto x \cdot y$. This yields the algebraic structure
$\AAA = (A; {\cdot})$, called a \emph{groupoid} (note that the term groupoid has
a different meaning in category theory).

Given such a groupoid, there are several ways of measuring how far our operation
is from being associative. For finite $\AAA$, a natural ``measure of nonassociativity"
is the number of triples $(a,b,c) \in A^3$ such that 
$(a \cdot b) \cdot c \neq a \cdot (b \cdot c)$. This notion was studied by  
A.~C.~Climescu \cite{Climescu-1947} as early as 1947, and later by T.~Kepka and M.~Trch 
in a long series of papers starting with \cite{KepkaTrch-1992}.
Another option is to count the minimum number of changes one has to make in the operation
table in order to make it associative \cite{KepkaTrch-1993}.

B.~Cs\'{a}k\'{a}ny suggested a third method, namely to look at how many of the 
identities that are consequences of associativity are (not) satisfied. 
If the operation is associative, then there is no need to use parentheses in a product 
$x_1 \cdot x_2 \cdot \ldots \cdot x_n$, as the result will be the same anyway,
but if the operation is not associative, then one must insert $n-2$ pairs of 
parentheses in order to make the product unambiguous.
The \emph{Catalan numbers} $C_{n-1} = \frac{1}{n} \binom{2n-2}{n-1}$ give the number
of ways of inserting parentheses (or round brackets) meaningfully, and each such 
\emph{bracketing} induces an $n$-variable function $A^n \to A$. For associative 
binary operations all these $n$-ary functions will be the same, but for arbitrary
operations we may get as many as $C_{n-1}$ functions. 
The \emph{associative spectrum} of $\AAA$ is the sequence 
$\{s_n(\AAA)\}_{n=1}^\infty$ that counts the number of different $n$-ary functions 
on $A$ arising from bracketings of the product $x_1 \cdot x_2 \cdot \ldots \cdot x_n$. 
If $\AAA$ is a semigroup (i.e., if $x \cdot y$ is associative), then $s_n(\AAA)=1$ for
all $n \in \IN$, and intuitively we can say that the faster the spectrum grows,
the less associative the operation is. 

The associative spectrum was introduced in 
\cite{CsaWal-2000}, and some basic properties and many examples of associative spectra
were presented. In particular, it was shown that the cross product and the implication
have a Catalan spectrum, hence they are as nonassociative as a binary operation can be.
We shall call such operations (groupoids) \emph{antiassociative}. The associative 
spectrum of the subtraction operation is given by $s_n=2^{n-2}$, thus subtraction
is somewhere between being associative and antiassociative. 
Examples of groupoids with constant and linear spectra were also given in \cite{CsaWal-2000}, furthermore, in \cite{LieWal-2009}
groupoids with polynomial spectra of arbitrary degrees were constructed.
It was also proved in \cite{LieWal-2009} that there exist a continuum of different
associative spectra (allowing infinite base sets, of course).
Similar questions were investigated in \cite{BraHobSil-2012,BraHobSil-2017,BraSil-2006},
where some of the earlier results were rediscovered  (with a different terminology).

In this paper we study associative spectra of certain  binary operations associated to graphs. Let us define a ``multiplication" on the vertices of a graph as follows: 
let $u \cdot v = u$ if there is an edge from $u$ to $v$ and let $u \cdot v = \infty$
otherwise (here $\infty$ is an external absorbing element). We define the arising 
\emph{graph algebras} more precisely in Section~\ref{sect prelim}, where we
also present the required background on bracketings and spectra.

For undirected graphs we obtain a full description of all possible associative spectra 
in Section~\ref{sect undirected}. It turns out that there are only three possibilities:
we have either $s_n=1$, $s_n=2^{n-2}$ or $s_n=C_{n-1}$. 
Note the sharp contrast between this result and the abundance of different 
(growth rates of) spectra presented in \cite{CsaWal-2000,LieWal-2009}.
In Theorem~\ref{thm:main-undirected} we also give explicit characterizations
of undirected graphs corresponding to each of the three spectra.

We determine antiassociative digraphs in Section~\ref{sect directed};
this together with the description of associative digraphs \cite{Poo-2000}
gives us at least a picture about the two extrema of the spectrum(!) of associative
spectra of digraphs. Finally, in Section~\ref{sect examples}
we compute the associative spectra of some concrete graphs such as cycles
and paths, and we also determine the spectra of graphs on two vertices.
A more detailed analysis of the associative spectra of general digraphs will be a topic of a forthcoming paper.

\section{Preliminaries}\label{sect prelim}

\subsection{General notation}

We denote by $\IN$ and $\IN_{+}$ the set of nonnegative integers and the set of positive integers, respectively.
For $a, b \in \IN$, let $[a,b] := \{i \in \IN \mid a \leq i \leq b\}$. (Thus $[a,b] = \emptyset$ if $a > b$.)
For $n \in \IN$, let $\nset{n} := [1,n] = \{1, \dots, n\}$.

\subsection{Directed graphs}

By a \emph{directed graph} (or \emph{digraph} or simply \emph{graph}) we mean a pair $G = (V, E)$, where $V = V(G)$ is a nonempty set of \emph{vertices} and $E = E(G) \subseteq V^2$ is a set of \emph{edges} (or the \emph{edge relation}).
A digraph $G' = (V', E')$ is a \emph{subgraph} of $G = (V,E)$ if $V' \subseteq V$ and $E' \subseteq E$; it is an \emph{induced subgraph} of $G$ if additionally $E' = E \cap (V' \times V')$.

If $e = (u,v) \in E$, then we say that $e$ is an edge from $u$ to $v$, and we sometimes denote this by $u \rightarrow v$.
In this case we also say that $u$ is an \emph{inneighbour} of $v$ and $v$ is an \emph{outneighbour} of $u$.
The \emph{outneigbourhood} of a vertex $u \in V(G)$, denoted by $\Nout[G]{u}$, is the set of all outneighbours of $u$ in $G$.
The concept of \emph{inneighbourhood} is defined analogously.
An edge of the form $(u,u)$ is called a \emph{loop} (on $u$) and sometimes denoted by $u\ \loopedge$.

A \emph{walk} of length $\ell$ from $u$ to $v$ in $G$ is a sequence $v_0, \dots, v_\ell$ of (not necessarily distinct) vertices such that $v_0 = u$, $v_\ell = v$, and there is an edge from each vertex to the next one (except for the last vertex, of course): $v_0 \rightarrow v_1 \rightarrow \dots \rightarrow v_\ell$.
If $v_0 = v_\ell$, then we say that the walk is \emph{closed.}
A \emph{path} (\emph{cycle}) is a (closed) walk in which the vertices are pairwise distinct (with the exception of the first and last vertex in case of a cycle).
A digraph without cycles is called \emph{acyclic.}

We say that a vertex $u$ is \emph{reachable} from $v$ if there exists a walk (equivalently, a path) from $v$ to $u$.
A pair of vertices $u$ and $v$ are said to be \emph{strongly connected} 
if each one of $u$ and $v$ is reachable from the other.
The relation of being strongly connected is an equivalence relation, and the induced subgraphs of its equivalence classes are called the \emph{strongly connected components} of $G$.
A digraph is \emph{strongly connected} if it has just one strongly connected component.
A one\hyp{}vertex graph with no edge is strongly connected (let us call this the \emph{trivial strongly connected graph});
apart from this trivial example, every vertex of a strongly connected digraph is contained in a cycle of nonzero length
(this includes the graph of one vertex with a loop on it).

A digraph with a symmetric edge relation is called an \emph{undirected graph.}
The strongly connected components of an undirected graph are called \emph{connected components.}
The \emph{underlying undirected graph} of a digraph $G = (V,E)$ is the undirected graph $(V,E')$, where the edge relation $E'$ equals the symmetric closure of $E$.

A \emph{tree} is an undirected graph in which any two vertices are connected by exactly one path.
A \emph{rooted directed tree} is a directed acyclic graph whose underlying undirected graph is a tree and that has a distinguished vertex, called a \emph{root}, from which all vertices are reachable.
Let $v$ be a vertex of a rooted directed tree $T$.
Unless $v$ is the root of $T$, it has a unique inneighbour, which is referred to as the \emph{parent} of $v$.
The outneighbours of $v$ are called \emph{children} of $v$.
A childless vertex is called a \emph{leaf.}
The vertices reachable from $v$ are called \emph{descendants} of $v$, and $v$ is called an \emph{ancestor} of any of its descendants.
The \emph{rooted induced subtree} of $T$ rooted at $v$, denoted by $T_v$, is the subgraph of $T$ induced by $v$ and all its descendants.

The \emph{depth} of a vertex $v$ in a rooted directed tree $T$ is the length of the (unique) path from the root to $v$, denoted by $d_T(v)$.
(Thus the root has depth $0$.)
The \emph{height} of $T$, denoted by $h(T)$, is the maximum of the depths of its vertices: $h(T) = \max \{d_T(v) \mid v \in V(T)\}$.

\subsection{Graph algebras}

Graph algebras were introduced by C.~R.~Shallon~\cite{Shallon}. We associate any digraph $G = (V,E)$ with an algebra $\GA{G} = (V \cup \{\infty\}; {\circ}, \infty)$ of type $(2,0)$, where $\infty$ is a new element distinct from the vertices, and the binary operation is defined by the following rule: for any $x, y \in V \cup \{\infty\}$,
\[
x \circ y :=
\begin{cases}
x, & \text{if $(x,y) \in E$,} \\
\infty, & \text{otherwise.}
\end{cases}
\]
The algebra $\GA{G}$ is called the \emph{graph algebra} of $G$.
Graph algebras provide a simple encoding of graphs as algebras, and
using this encoding, the algebraic properties of the graph algebra $\GA{G}$ can be seen as properties of the graph $G$ itself.

We are particularly interested in the satisfaction of identities by graph algebras.
Recall that a \emph{term} is, informally speaking, a well\hyp{}formed string comprising variables and function symbols from the language of algebras under consideration.
An \emph{identity} is an ordered pair $(t,t')$ of terms, usually written as $t \approx t'$.
An algebra $\mathbb{A}$ \emph{satisfies} an identity $t \approx t'$ if for all assignments of values to the variables occurring in $t$ and $t'$, the two terms get the same value when the function symbols are interpreted as the fundamental operations of $\mathbb{A}$.
An identity $t \approx t'$ is \emph{trivial} if $t = t'$.
Trivial identities are clearly satisfied by all algebras (of the given type).
For further details, see, e.g., \cite{DenWis}.

Let $t$ be a term in the language of graph algebras.
Denote by $\var(t)$ the set of variables occurring in $t$ and by $L(t)$ the leftmost variable occurring in $t$.
We say that $t$ is \emph{trivial} if it contains an occurrence of the constant symbol $\infty$; otherwise $t$ is \emph{nontrivial.}
Nontrivial terms are thus just groupoid terms.
To any nontrivial term $t$, we can associate a digraph $G(t) = (V,E)$, where $V = \var(t)$, and $(x_i, x_j) \in E$ if and only if $t$ has a subterm $(t_1 \circ t_2)$ with $L(t_1) = x_i$ and $L(t_2) = x_j$.

The following result is very helpful for determining whether a graph algebra satisfies an identity.

\begin{proposition}[{P\"{o}schel, Wessel \cite[Proposition~1.5(2)]{PosWes-1987}}]
\label{prop:phi}
Let $G = (V,E)$ be a digraph, and let $\GA{G}$ denote the corresponding graph algebra.
Let $t$ and $t'$ be nontrivial terms in the language of graph algebras, and assume that $\var(t) = \var(t')$ and $L(t) = L(t')$.
Then the following conditions are equivalent:
\begin{enumerate}[label=\upshape{(\roman*)}]
\item $\GA{G}$ satisfies $t \approx t'$;
\item for every map $\varphi \colon \var(t) \to V$, we have that $\varphi$ is a homomorphism of $G(t)$ into $G$ if and only if $\varphi$ is a homomorphism of $G(t')$ into $G$.
\end{enumerate}
\end{proposition}

\subsection{Associative spectra}

Let $B_n$ denote the set of \emph{bracketings} of size $n$, i.e., groupoid terms obtained from the string $x_1 \cdot x_2 \cdot \ldots \cdot x_n$ by inserting parentheses appropriately.
The number of bracketings of size $n$ is given by the $(n-1)$\hyp{}st Catalan number $C_{n-1} = \frac{1}{n} \binom{2n-2}{n-1}$.
If $\AAA = (A; {\cdot})$ is a groupoid, then the equational theory of $\AAA$ induces an equivalence relation $\sigma_n(\AAA)$ on $B_n$.
The sequence $\{\sigma_n(\AAA)\}_{n=1}^\infty$ is called the \emph{fine associative spectrum} of $\AAA$.
The \emph{associative spectrum} of $\AAA$ is the sequence $\{s_n(\AAA)\}_{n=1}^\infty$ of natural numbers defined by $s_n(\AAA) := \card{B_n / \sigma_n(\AAA)}$.
Equivalently, $s_n(\AAA)$ is the number of distinct term operations of $\AAA$ induced by the bracketings of size $n$.
Intuitively, the faster the associative spectrum grows, the less associative the operation is.
The groupoid $\AAA$ is a semigroup if and only if $s_n(\AAA) = 1$ for all $n \in \IN$.
On the other extreme we have the \emph{antiassociative groupoids} whose associative spectrum is given by the Catalan numbers: $s_n(\AAA) = \card{B_n} = C_{n-1}$.
These groupoids do not satisfy any nontrivial identity of the form $t_1 \approx t_2$ with $t_1, t_2 \in B_n$.

Since there exists only one bracketing of size $1$, namely $x_1$, and of size $2$, namely $(x_1 x_2)$, it is clear that $s_1(\AAA) = s_2(\AAA) = 1$ for every groupoid $\AAA$.
Therefore we may always assume that $n \geq 3$ when we consider bracketings of size $n$ or the $n$\hyp{}th component of an associative spectrum.

\subsection{DFS trees}

It turns out that the graphs associated with bracketings are particularly nice; they are rooted directed trees of a very special form.

\begin{figure}
\tikzset{every node/.style={circle,draw,inner sep=1.5,fill=black}, every path/.style={->,>=stealth,thick}, dfspath/.style={>->,>=stealth,thin,dotted}}
\scalebox{1}{
\begin{tikzpicture}[scale=1, baseline=(x1)]
\node[label={[label distance=1mm] below:{$x_1$}}] (x1) at (0,0) {};
\node[label={[label distance=1mm] below left:{$x_2$}}] (x2) at (-2,1) {};
\node[label={[label distance=1mm] above left:{$x_3$}}] (x3) at (-3,2) {};
\node[label={[label distance=1mm] above right:{$x_4$}}] (x4) at (-1,2) {};
\node[label={[label distance=0.7mm] above:{$x_5$}}] (x5) at (0,1) {};
\node[label={[label distance=1mm] below right:{$x_6$}}] (x6) at (1,1) {};
\node[label={[label distance=1mm] below right:{$x_7$}}] (x7) at (2,2) {};
\node[label={[label distance=0.7mm] above:{$x_8$}}] (x8) at (2,3) {};
\path (x1) edge (x2);
\path (x2) edge (x3);
\path (x2) edge (x4);
\path (x1) edge (x5);
\path (x1) edge (x6);
\path (x6) edge (x7);
\path (x7) edge (x8);
\draw [dfspath]
($(x1)+({180+atan(2)}:0.2)$)
-- ($(x2)+({180+atan(2)}:0.2)$)
arc ({180+atan(2)}:225:0.2)
-- ($(x3)+(225:0.2)$)
arc (225:45:0.2)
-- ($(x2)+(45:0.2)+(135:0.3)$)
arc (225:315:0.1)
-- ($(x4)+(135:0.2)$)
arc (135:-45:0.2)
-- ($(x2)+({(45-atan(0.5))/2}:{0.3/sin((45+atan(0.5))/2)})+(135:0.1)$)
arc (135:{180+atan(2)}:0.1)
-- ($(x1)+({90+(atan(2)/2)}:{0.3/sin(atan(2)/2)})+({180+atan(2)}:0.1)$)
arc ({180+atan(2)}:360:0.1)
-- ($(x5)+(180:0.2)$)
arc (180:0:0.2)
-- ($(x1)+(67.5:{0.25/sin(22.5)})+(180:0.05)$)
arc (180:315:0.05)
-- ($(x7)+(157.5:{0.3/sin(67.5)})+(315:0.1)$)
arc (315:360:0.1)
-- ($(x8)+(180:0.2)$)
arc (180:0:0.2)
-- ($(x7)+(0:0.2)$)
arc (0:-45:0.2)
-- ($(x1)+(315:0.2)+(225:0.1)$);
\end{tikzpicture}
}
\caption{The DFS tree of the bracketing $((x_1((x_2x_3)x_4))x_5)(x_6(x_7x_8))$.}
\label{fig:DFSpath}
\end{figure}
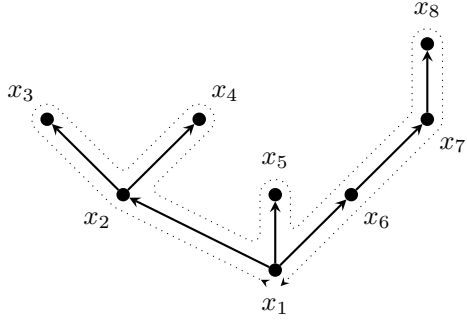

\begin{definition}
\label{def:DFS}
A \emph{DFS tree} of size $n$ is 
a rooted directed tree $T$ on the vertex set $X_n := \{x_1, x_2, \dots, x_n\}$ that has root $x_1$ and for every vertex $x_i \in X_n$, the induced subtree $T_{x_i}$ has vertex set of the form $X_{[i,i']} := \{x_j \mid j \in [i,i']\}$ for some $i' \in \nset{n}$ with $i' \geq i$.
\end{definition}

The name ``DFS tree'' stems from the fact that the vertices are labeled in an order in which they may be traversed by the depth\hyp{}first search (DFS) (see \cite[Section~22.3]{CLRS}) starting from the root.
Figure~\ref{fig:DFSpath} shows the DFS tree $G(t)$ of size $8$ that corresponds to the bracketing
$t=((x_1((x_2x_3)x_4))x_5)(x_6(x_7x_8))$. The dotted line shows the walk traversed by
the depth first search (using the convention that the search continues always with the leftmost
unvisited child). Note that the order of first occurrence of the vertices along this walk is
$x_1,x_2,\dots,x_8$.

\begin{lemma}[{cf.\ \cite[Theorem~22.7 (Parenthesis theorem)]{CLRS}}]
\label{lem:DFS}
Let $T$ be a rooted directed tree on $X_n$.
The following are equivalent.
\begin{enumerate}[label={\upshape (\roman*)}]
\item\label{lem:DFS:intervals}
$T$ is a DFS tree.
\item\label{lem:DFS:order} The sequence $x_1, x_2, \dots, x_n$ is a possible order in which the vertices of $T$ may be traversed by the depth\hyp{}first search starting from the root.
\end{enumerate}
\end{lemma}

\begin{proof}
\ref{lem:DFS:intervals}~$\Rightarrow$~\ref{lem:DFS:order}:
Assume that $T$ satisfies condition \ref{lem:DFS:intervals}.
Condition \ref{lem:DFS:order} will follow if we prove that for each vertex $x_i \in X_n$, the vertices of the rooted induced subtree $T_{x_i}$ (by our assumption $V(T_{x_i}) = X_{[i,i']}$ for some $i' \geq i$) may be traversed by the depth\hyp{}first search in the order $x_i, x_{i+1}, \dots, x_{i'}$.
We proceed by induction on the height of subtrees.
The claim obviously holds for rooted induced subtrees of height $0$.
Assume that the claim holds for rooted induced subtrees of height at most $k$, and let $x_i \in X_n$ be a vertex such that $h(T_{x_i}) = k + 1$.
Let $x_{i_1}, x_{i_2}, \dots, x_{i_\ell}$ be the children of $x_i$ in $T$ with $i_1 < i_2 < \dots < i_\ell$.
By condition \ref{lem:DFS:intervals}, for each $s \in \nset{\ell}$, $V(T_{x_{i_s}}) = X_{[i_s, i'_s]}$ for some $i'_s \geq i_s$; in fact $i'_s = i_{s+1} - 1$ for $1 \leq s < \ell$, $i_1 = i + 1$, and $V(T_{x_i}) = X_{[i,i'_\ell]}$.
By the induction hypothesis, the vertices of $T_{x_{i_s}}$ may be traversed by the depth\hyp{}first search in the order $x_{i_s}, x_{i_s + 1}, \dots, x_{i'_s}$;
consequently, the vertices of $T_{x_i}$ may be traversed in the order $x_i, x_{i_1}, x_{i_1 + 1}, \dots, x_{i'_1}, x_{i_2}, x_{i_2 + 1}, \dots, x_{i'_2}, \dots, x_{i_\ell}, x_{i_\ell + 1}, \dots, x_{i'_\ell}$, that is, in the order $x_i, x_{i+1}, \dots, x_{i'_\ell}$.

\ref{lem:DFS:order}~$\Rightarrow$~\ref{lem:DFS:intervals}:
Assume that $T$ satisfies condition \ref{lem:DFS:order}.
For any $x_i \in X_n$, $x_i$ is the first vertex in $T_{x_i}$ visited by the depth\hyp{}first search, all vertices of $T_{x_i}$ are traversed before the depth\hyp{}first search continues with vertices not belonging to $T_{x_i}$, and once the depth\hyp{}first search leaves the subtree $T_{x_i}$ it will never return to it.
Consequently, condition \ref{lem:DFS:intervals} clearly holds.
\end{proof}

Bracketings of size $n$ are in a one\hyp{}to\hyp{}one correspondence with DFS trees of size $n$.

\begin{lemma}[{cf.\ Kiss~\cite[Lemma~2]{Kiss}}]
Let $n \in \IN$.
\begin{enumerate}[label={\upshape{(\alph*)}}]
\item\label{br-tr} For any bracketing $t \in B_n$, the graph $G(t)$ is a DFS tree of size $n$.
\item\label{tr-br} Conversely, for every DFS tree $T$ of size $n$, there is a unique bracketing $t \in B_n$ such that $G(t) = T$.
\end{enumerate}
\end{lemma}

\begin{proof}
\ref{br-tr}
Let $t \in B_n$.
Then $G(t)$ is a graph on $X_n$ by definition.
We will prove by induction on the structure of terms that for every subterm $t'$ of $t$, the graph $G(t')$ is a directed tree on $\var(t')$ with root $L(t')$ such that for every $x_i \in \var(t')$, the subtree of $G(t')$ rooted at $x_i$ has vertex set of the form $X_{[i,i']}$ for some $i' \in \nset{n}$ with $i' \geq i$.
The claim obviously holds for any subterm of the form $t' = x_i \in X_n$.
Let now $t' = (t_1 \circ t_2)$, and assume that the claim holds for the subterms $t_1$ and $t_2$.
By the induction hypothesis, for $\ell \in \{1,2\}$, $G(t_\ell)$ is a directed tree on $\var(t_\ell)$ with root $L(t_\ell)$; moreover, $\var(t_\ell) = X_{[p_\ell, q_\ell]}$, where $p_\ell = L(t_\ell)$ and $q_\ell \geq p_\ell$.
In fact, $q_1 = p_2 - 1$.
Since $\var(t_1) \cap \var(t_2) = \emptyset$, $G(t')$ is obtained by adding the edge $L(t_1) \rightarrow L(t_2)$ to the disjoint union of $G(t_1)$ and $G(t_2)$; the resulting graph is a directed tree on $\var(t') = X_{[p_1, q_2]}$.
Moreover, given a vertex $x_i \in \var(t')$, we have that the subtree of $G(t')$ induced by $x_i$ is identical to the one induced by $x_i$ in $G(t_1)$ if $x_i \in \var(t_1) \setminus \{x_{p_1}\}$ and identical to the one induced by $x_i$ in $G(t_2)$ if $x_i \in \var(t_2)$; by the induction hypothesis, the subtree has the desired form.

\ref{tr-br}
For the purpose of this proof, we relax the notions of bracketing and DFS tree so as to allow variable or vertex sets of the form $X_{[a,b]}$.
Let $a, b \in \IN$ with $a \leq b$, and let $n := b - a + 1$.
A term $t$ with $\var(t) = X_{[a,b]}$ is an \emph{$[a,b]$\hyp{}bracketing} if $t$ can be obtained from some $t' \in B_n$ by replacing each variable $x_i$ by $x_{a+i-1}$, $1 \leq i \leq n$.
Similarly, a rooted directed tree $T$ on $X_{[a,b]}$ is an \emph{$[a,b]$\hyp{}DFS tree} if there is a DFS tree $T'$ of size $n$ such that the map $x_i \mapsto x_{a+i-1}$ is an isomorphism $T' \to T$.

We show that for any $a, b \in \IN$ with $a \leq b$, it holds that for every $[a,b]$\hyp{}DFS tree $T$, there exists a unique $[a,b]$\hyp{}bracketing $t$ such that $G(t) = T$.
We proceed by induction on the length $b-a$ of the interval $[a,b]$.
The claim is obvious for $b - a = 0$, i.e., $a = b$.
Assume that the claim holds whenever $b - a \leq k$.
Let now $a$ and $b$ be such that $b - a = k + 1$, and let $T$ be an $[a,b]$\hyp{}DFS tree.
Let $x_{i_1}, x_{i_2}, \dots, x_{i_\ell}$ be the children of the root vertex $x_a$, and assume that $i_1 < i_2 < \dots < i_\ell$.
Then $T - T_{x_{i_\ell}}$ is an $[a, i_\ell - 1]$\hyp{}DFS tree and $T_{x_{i_\ell}}$ is an $[i_\ell, b]$\hyp{}DFS tree, so by the induction hypothesis there exist a unique $[a, i_\ell - 1]$\hyp{}bracketing $r$ such that $T - T_{x_{i_\ell}} = G(r)$ and a unique $[i_\ell, b]$\hyp{}bracketing $s$ such that $T_{x_{i_\ell}} = G(s)$.
Then $t := (r \circ s)$ is an $[a,b]$\hyp{}bracketing and it is easy to see that $G(t) = T$ because $L(r) = x_a$ and $L(s) = x_{i_\ell}$.
This proved existence.
As for uniqueness, assume $t'$ is another $[a,b]$\hyp{}bracketing such that $G(t') = T$.
Since $x_a \rightarrow x_{i_\ell}$ is an edge in $T$, $t'$ must contain a subterm of the form $(r' \circ s')$ where $L(r') = x_a$, $L(s') = x_{i_\ell}$.
Then $\var(r') = X_{[a, i_\ell - 1]}$, so $G(r')$ is the subtree of $T$ with vertex set $X_{[a, i_\ell - 1]}$, that is $G(r') = T - T_{x_{i_\ell}} = G(r)$.
Observe that $t'$ contains no subterm of the form $(r' \circ s') \circ s''$ (otherwise $L(s'') =: x_p$ would be a child of $x_a$ with $p > i_\ell$, contradicting the choice of $i_\ell$).
Consequently $\var(s') = X_{[i_\ell, b]}$, so $G(s') = T_{x_{i_\ell}} = G(s)$.
By the induction hypothesis $r = r'$ and $s = s'$, so $t = (r \circ s) = (r' \circ s') = t'$.
\end{proof}

\begin{proposition}
\label{prop:unique-depth-seq}
DFS trees are uniquely determined by their depth sequences:
if $T$ and $T'$ are DFS trees of size $n$ such that $d_T(x_i) = d_{T'}(x_i)$ for all $i \in \{1, \dots, n\}$, then $T = T'$.
\end{proposition}

\begin{proof}
Suppose, to the contrary, that DFS trees $T$ and $T'$ satisfy $d_T(x_i) = d_{T'}(x_i)$ for all $i \in \{1, \dots, n\}$ but $T \neq T'$.
Then there exists a vertex $x_d \in X_n$ such that its parent $x_p$ in $T$ is distinct from its parent $x_q$ in $T'$.
Assume without loss of generality that $p < q$.
Since $d_T(x_d) = d_{T'}(x_d)$, we also have $d_{T'}(x_p) = d_T(x_p) = d_T(x_d) - 1 = d_{T'}(x_d) - 1 = d_{T'}(x_q) = d_T(x_q)$.
It follows from this that $x_d \in T_{x_p}$ and $x_q \notin T_{x_p}$; therefore $p < d < q$ by Definition~\ref{def:DFS}.
On the other hand, $x_d \in T'_{x_q}$; therefore $q < d$.
We have reached a contradiction.
\end{proof}

A sequence $(d_1,\dots,d_n)$ of nonnegative integers is called a
\emph{zag sequence}\footnote{Another, perhaps more telling name 
suggested by B\'{e}la Cs\'{a}k\'{a}ny
is \emph{Sisyphus sequence}: zag sequences can increase only gradually,
in steps of $1$, but they can decrease arbitrarily.} if
\begin{equation}
\label{eq:zag}
d_1=0,\ d_2=1,\text{ and }1 \leq d_{i+1} \leq d_i + 1\text{ for all } i\in\{1,\dots,n-1\}.    
\end{equation}
This notion was introduced in \cite{CsaWal-2000}, where bracketings were
represented by binary trees instead of DFS trees. 
(See also Exercise~19(u) in \cite{Stanley-Enum2}.)
The depth of a vertex in
a DFS tree is the same as the so-called ``right depth'' of the corresponding
vertex in the binary tree representing the same bracketing.
Therefore, 2.8 of \cite{CsaWal-2000} implies that depth sequences of
DFS trees are in a one\hyp{}to\hyp{}one correspondence with zag sequences.
We include the easy proof of this fact for the sake of self-containedness.

\begin{proposition}
\label{prop:depth-zag-seq}
A sequence $(d_1,\dots,d_n)$ of nonnegative integers is the depth sequence of
a DFS tree of size $n$ if and only if it is a zag sequence.
\end{proposition}
\begin{proof}
Necessity is clear: if $x_{i+1}$ is a child of $x_i$ in a DFS tree $T$, 
then $d_T(x_{i+1})=d_T(x_i)+1$; otherwise $x_{i+1}$ is a child of one of the
ancestors of $x_i$, hence $d_T(x_{i+1}) \leq d_T(x_i)$.
We prove sufficiency by induction on $n$. The case $n=1$ is trivial, so 
let $n\geq 2$, and assume that every zag sequence of length less than $n$
is the depth sequence of a DFS tree (which is unique, by Proposition~\ref{prop:unique-depth-seq}).
Let $(d_1,\dots,d_n)$ be a zag sequence, and let $d_k=1$ be the last occurrence
of $1$ in the sequence (possibly $k=2$). Then $d_1,\dots,d_{k-1}$ and
$d_k-1,\dots,d_n-1$ are zag sequences of length less than $n$, hence, by
our induction hypothesis, they are depth sequences of DFS trees 
$T_1$ (of size $k-1$) and $T_2$ (of size $n-k+1$), respectively.
Let us form the disjoint union of $T_1$ and $T_2$ after applying the
renaming $x_i \mapsto x_{i+k-1}$ to the vertices of $T_2$.
Now if we add an edge from $x_1$ (the root of $T_1$) to 
$x_k$ (the new root of $T_2$), then we obtain a DFS tree of size $n$
with depth sequence $(d_1,\dots,d_n)$.
\end{proof}

\begin{remark}
\label{rem:lattice paths}
Zag sequences of length $n$ can be visualized as lattice paths from the origin to
to the line $x=n-1$ using steps $(1,1),(1,0),(1,-1),(1,-2),\dots$.
Another family of lattice paths is also closely related to bracketings and
DFS trees. A \emph{Dyck path} of semilength $n$ is a lattice path
from $(0,0)$ to $(2n,0)$ consisting of $n$ up-steps $U=(1,1)$ and $n$ down-steps $D=(1,-1)$
in such a way that the path never goes below the $x$ axis.
To construct the Dyck path corresponding to a DFS tree $T$, let us draw
$T$ in such a way that all edges point upwards, and the children of every
vertex are drawn in increasing order (of their subscripts) from left to right.
(All DFS trees in this paper are drawn using this convention.)
Let us follow the depth\hyp{}first search on $T$, including the
backtracking steps, returning to the root in the end 
(see the dotted line in Figure~\ref{fig:DFSpath}). 
For each step, we add an up-step $U$ or a down-step $D$ to our lattice path starting 
at the origin according to whether we are moving upwards or downwards in the tree.
(See Figure~\ref{fig:Dyck} for the Dyck path corresponding to 
the DFS tree of Figure~\ref{fig:DFSpath}. The first occurrence
of each vertex during the depth\hyp{}first search  is labelled on the diagram.)
This way we obtain a bijection from the set of DFS trees of size $n$ to the set of Dyck paths of semilength $n-1$. 
A \emph{wormderful} explanation of this bijection
is presented in \cite[p.~10]{Stanley-Catalan},
where this process is actually used
to define the depth\hyp{}first order.
\end{remark}

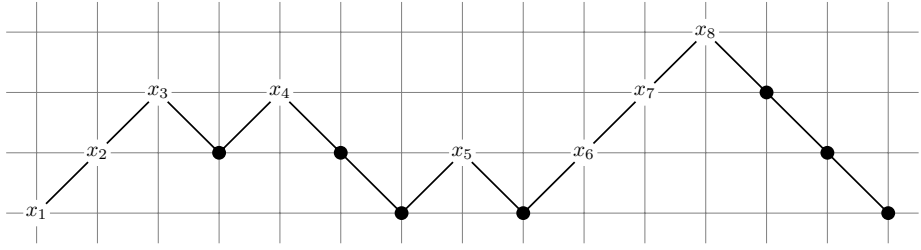
\begin{figure}
\tikzset{
    every node/.style={circle,inner sep=0.1mm,fill=white},
    dyckline/.style={-,thick,color=black,fill=black},
    dyckdot/.style={radius=0.1cm},
    dyckdot0/.style={radius=0.0cm},
    }
\scalebox{0.8}{
\begin{tikzpicture}[scale=1]
\draw [help lines] (-0.5,-0.5) grid (14.5,3.5);
\filldraw[dyckline] 
      ++(0,0) circle [dyckdot0]
    --++(1,1) circle [dyckdot0]
    --++(1,1) circle [dyckdot0]
    --++(1,-1) circle [dyckdot]
    --++(1,1) circle [dyckdot0]
    --++(1,-1) circle [dyckdot]
    --++(1,-1) circle [dyckdot]
    --++(1,1) circle [dyckdot0]
    --++(1,-1) circle [dyckdot]
    --++(1,1) circle [dyckdot0]
    --++(1,1) circle [dyckdot0]
    --++(1,1) circle [dyckdot0]
    --++(1,-1) circle [dyckdot]
    --++(1,-1) circle [dyckdot]
    --++(1,-1) circle [dyckdot];
\node (x1) at (0,0) {$x_1$};
\node (x2) at (1,1) {$x_2$};
\node (x3) at (2,2) {$x_3$};
\node (x4) at (4,2) {$x_4$};
\node (x5) at (7,1) {$x_5$};
\node (x6) at (9,1) {$x_6$};
\node (x7) at (10,2) {$x_7$};
\node (x8) at (11,3) {$x_8$};
\end{tikzpicture}
}
\caption{The Dyck path of the DFS tree of Figure~\ref{fig:DFSpath}.}
\label{fig:Dyck}
\end{figure}

\subsection{Collapsing maps and a few lemmas}

\begin{definition}
Let $T$ be a DFS tree of size $n$, and let $G$ be a digraph.
If $h = h(T)$ and $W \colon v_0 \rightarrow v_1 \rightarrow \dots \rightarrow v_h$ is a walk in $G$, then the mapping $\varphi \colon X_n \to V(G)$, $x_i \mapsto v_{d_T(x_i)}$ is clearly a homomorphism of $T$ into $G$.
Similarly, if $C \colon u_0 \rightarrow u_1 \rightarrow \dots \rightarrow u_{\ell - 1} \rightarrow u_0$ is a closed walk in $G$ with $\ell \geq 1$, then the mapping $\psi \colon X_n \to V(G)$, $x_i \mapsto v_{d_T(x_i) \bmod \ell}$ is a homomorphism of $T$ into $G$.
Such homomorphisms $\varphi$ and $\psi$ are referred to as \emph{collapsing maps} of $T$ on $W$ and $C$, respectively, and we say that the DFS tree $T$ is \emph{collapsed} on the walk $W$ (on the closed walk $C$) by $\varphi$ (by $\psi$).

We will often specify homomorphisms of DFS trees by giving a piecewise definition in which each piece is a collapsing map of a subgraph.
In particular, if $T$ is a DFS tree of size $n$, $x_d \in X_n$, $s = d_T(x_d) > 0$, $h = h(T)$, $h' = h(T_{x_d})$, $W \colon v_0 \rightarrow v_1 \rightarrow \dots \rightarrow v_s \rightarrow \dots \rightarrow v_h$ is a walk in $G$, $v_{s-1} \rightarrow u_0$ is an edge, and $W'$ is either a walk $u_0 \rightarrow u_1 \rightarrow \dots \rightarrow u_{h'}$ or a closed walk $u_0 \rightarrow u_1 \rightarrow \dots \rightarrow u_\ell \rightarrow u_0$, then the mapping $\varphi \colon X_n \to V(G)$ that collapses $T \setminus T_{x_d}$ on $W$ and $T_{x_d}$ on $W'$
is a homomorphism of $T$ into $G$, and we will refer to $\varphi$ as the \emph{collapsing map} of $(T,x_d)$ on $(W,W')$, and we say that $(T,x_d)$ is \emph{collapsed} on $(W,W')$ by $\varphi$.
\end{definition}

With the help of collapsing maps and Proposition~\ref{prop:phi}, we can derive conditions for the edges of a digraph satisfying a bracketing identity.
Let us illustrate this with a few examples that will serve as helpful tools later.

\begin{lemma}
\label{lem:TT'edges}
Let $t, t' \in B_n$, $t \neq t'$, $T := G(t)$, $T' := G(t')$, $h := h(T)$, and
let $G$ be a digraph such that $\GA{G}$ satisfies the identity $t \approx t'$.
If $W \colon v_0 \rightarrow v_1 \rightarrow \dots \rightarrow v_h$ is a walk in $G$, then $(v_{d_T(a)}, v_{d_T(b)}) \in E(G)$ for every $(a,b) \in E(T')$.
\end{lemma}

\begin{proof}
The collapsing map $\varphi$ of $T$ to $W$ is a homomorphism of $T$ into $G$, so it is also a homomorphism of $T'$ into $G$ by Proposition~\ref{prop:phi}.
Consequently, for every edge $(a,b)$ of $T'$, we have $(v_{d_T(a)}, v_{d_T(b)}) = (\varphi(a), \varphi(b)) \in E(G)$.
\end{proof}

\begin{lemma}
\label{lem:loops}
Let $t, t' \in B_n$, $t \neq t'$, $T := G(t)$, $T' := G(t')$, and
let $G$ be a digraph such that $\GA{G}$ satisfies the identity $t \approx t'$.
Then the following statements hold.
\begin{enumerate}[label={\upshape{(\alph*)}}]
\item\label{lem:loops:refl-symm}
If $u, v \in V(G)$, $\{(u,u), (u,v)\} \subseteq E(G)$, and $G$ contains arbitrarily long walks with initial vertex $v$, then $\{(v,u), (v,v)\} \subseteq E(G)$.
\item\label{lem:loops:trans}
If $u, v, w \in V(G)$ and $\{(u,u), (v,v), (w,w), (v,u), (v,w)\} \subseteq E(G)$, then $\{(u,w), (w,u)\} \subseteq E(G)$.
\item\label{lem:loops:even-symm}
If $d_T(x_i) \equiv d_{T'}(x_i) \pmod{2}$ for all $x_i \in X_n$, $u, v, w \in V(G)$, $\{(u,v),\linebreak[0] (v,u),\linebreak[0] (v,w)\} \subseteq E(G)$ and $G$ contains arbitrarily long walks with initial vertex $w$, then $(w,v) \in E(G)$.
\item\label{lem:loops:even-shortcut}
If $d_T(x_i) \equiv d_{T'}(x_i) \pmod{2}$ for all $x_i \in X_n$, $u, v, u', v' \in V(G)$, $\{(u,v),\linebreak[0] (v,u),\linebreak[0] (u,v'), (v',u), (v,u'), (u',v)\} \subseteq E(G)$, then $\{(u',v'), (v',u')\} \subseteq E(G)$.
\end{enumerate}
\end{lemma}

\begin{proof}
Since $t \neq t'$, there exists a vertex $x_d$ that has distinct parents in $T$ and $T'$, say $x_p$ and $x_q$, respectively.
We have $p < d$ and $q < d$, and, changing the roles of $T$ and $T'$ if necessary, we may assume that $p < q < d$.
By Definition~\ref{def:DFS},
$x_q, x_d \in V(T_{x_p})$ but $x_q \notin V(T_{x_d})$.
Let $h := h(T)$, $s := d_T(x_d)$, $r := d_T(x_q)$. Note that $d_T(x_p) = s-1$ and $r \geq s$.

For each statement, we are going to provide suitable walks $W \colon v_0 \rightarrow v_1 \rightarrow \dots \rightarrow v_h$ and $W' \colon u_0 \rightarrow u_1 \rightarrow \cdots$ in $G$, with $v_{s-1} \rightarrow u_0$ being an edge, and we consider the collapsing map $\varphi$ of $(T,x_d)$ on $(W,W')$, which is a homomorphism of $T$ into $G$.
By Proposition~\ref{prop:phi}, $\varphi$ is also a homomorphism of $T'$ into $G$.
Since $(x_q, x_d) \in E(T')$, we obtain the desired edge $(\varphi(x_q), \varphi(x_d)) = (v_r, u_0) \in E(G)$.

\ref{lem:loops:refl-symm}
We obtain the edge $v \rightarrow u$ by letting
$W$ be the walk starting with $r$ occurrences of $u$, followed by a walk of length $h-r$ starting at $v$,
and letting $W'$ be the cycle $u \rightarrow u$.
We obtain the edge $v \rightarrow v$ by letting
$W$ be as above
and letting $W'$ be a sufficiently long walk starting at $v$.

\ref{lem:loops:trans}
We obtain the edge $u \rightarrow w$ by letting
$W \colon v \rightarrow \dots \rightarrow v \rightarrow u \rightarrow \dots \rightarrow u$ with $r$ occurrences of $v$ and $h-r+1$
occurrences of $u$
and $W' \colon w \rightarrow w$.
By swapping $u$ with $w$ in the above, we obtain also the edge $w \rightarrow u$.

\ref{lem:loops:even-symm}
We obtain the edge $w \rightarrow v$ by letting
$W \colon v_0 \rightarrow v_1 \rightarrow \dots \rightarrow v_h$ be the walk in which $v_0, \dots, v_{r-1}$ alternate between vertices $u$ and $v$ such that $v_{r-1} = v$, followed by the vertices of a walk of length $h-r$ starting at $w$,
and letting $W'$ be the cycle $v \rightarrow u \rightarrow v$.
Note that $s-1 = d_T(x_p) = d_T(x_d) - 1 \equiv d_{T'}(x_d) - 1 = d_{T'}(x_q) \equiv d_T(x_q) = r \pmod{2}$, so $v_{s-1} = u$ and $v_{s-1} \rightarrow u_0$ is indeed an edge in $G$.

\ref{lem:loops:even-shortcut}
We obtain the edge $u' \rightarrow v'$ by letting
$W \colon v_0 \rightarrow v_1 \rightarrow \dots \rightarrow v_h$ be the walk with
$v_i := u$ for $i \equiv r \pmod{2}$, $i \neq r$, 
$v_r := u'$, and
$v_i := v$ for $i \not\equiv r \pmod{2}$,
and letting $W'$ be the cycle $v' \rightarrow u \rightarrow v'$.
Note that we have $s-1 \equiv r \pmod{2}$ as above, so $v_{s-1} = u$ and $v_{s-1} \rightarrow u_0$ is indeed an edge in $G$.
\end{proof}

\begin{remark}
There exist arbitrarily long walks with initial vertex $v$ if, for example, $v$ lies on a cycle, or there is a path from $v$ to a vertex $v'$ that lies on a cycle.
\end{remark}

\section{Associative spectra of graph algebras of undirected graphs}
\label{sect undirected}

It is relatively easy to determine the associative spectra of graph algebras of undirected graphs.
It turns out that there are only three distinct possibilities: the sequences of all $1$'s, powers of $2$, and Catalan numbers.
Undirected graphs are classified into these three types in Theorem~\ref{thm:main-undirected}.

As we will see, a key criterion for the classification of pairs of distinct DFS trees of size $n$ is whether
the depths of each vertex in the two trees are congruent modulo $2$.

\begin{lemma}
\label{lem:H2}
Let $\sim$ be the equivalence relation on $B_n$ that relates $t$ and $t'$ if and only if
$d_{G(t)}(x_i) \equiv d_{G(t')}(x_i) \pmod{2}$ for all $x_i \in X_n$.
Then $\card{B_n / {\sim}} = 2^{n-2}$ for $n \geq 2$.
\end{lemma}

\begin{proof}
The \emph{depth sequence modulo $2$} of a bracketing $t \in B_n$ is the tuple
$d_{t,2} := (d_1, d_2, \dots, d_n)$, where $d_i := d_{G(t)}(x_i) \bmod{2}$.
We clearly have $d_{t,2} \in \{0\} \times \{1\} \times \{0,1\}^{n-2}$, because $x_1$ and $x_2$ always have depths $0$ and $1$, respectively.
On the other hand, every tuple $(d_1, d_2, \dots, d_n) \in \{0\} \times \{1\} \times \{0,1\}^{n-2}$ is the depth sequence modulo $2$ of some bracketing $t \in B_n$, which we can build as follows.
The vertices $x_1$ and $x_2$ must have depths $0$ and $1$, respectively.
For $j = 2, \dots, n$,
if $d_j \neq d_{j-1}$, then we add $x_j$ as a child of $x_{j-1}$;
if $d_j = d_{j-1}$, then we add $x_j$ as a child of the unique parent of $x_{j-1}$.
It is now obvious that $\card{B_n / {\sim}} = \card{\{0\} \times \{1\} \times \{0,1\}^{n-2}} = 2^{n-2}$.
\end{proof}

\begin{lemma}
\label{lem:complete-bipartite}
Let $K$ be an undirected connected graph with no loops.
Assume that for all vertices $a$, $b$, $c$, $d$ of $K$ it holds that if $a \rightarrow b \rightarrow c \rightarrow d$ is a walk in $K$, then $a \rightarrow d$ is an edge.
Then $K$ is complete bipartite.
\end{lemma}

\begin{proof}
Suppose, to the contrary, that $K$ is not bipartite.
Then $K$ has a cycle of odd length $m \geq 3$, say $v_1 \rightarrow v_2 \rightarrow \dots \rightarrow v_m \rightarrow v_1$.
By applying our assumption to the walk $v_{m-2} \rightarrow v_{m-1} \rightarrow v_m \rightarrow v_1$, we get the edge $v_{m-2} \rightarrow v_1$;
hence $v_1 \rightarrow v_2 \rightarrow \dots \rightarrow v_{m-2} \rightarrow v_1$ is a cycle of length $m-2$ in $K$.
Repeating this argument, we eventually arrive at a cycle of length $1$.
This contradicts the fact that $K$ has no loops.

We have established that $K$ must be bipartite.
It remains to show that $K$ is complete bipartite.
Let $B_1$, $B_2$ be a bipartition of $K$, and let $x \in B_1$, $y \in B_2$.
We want to show that $x \rightarrow y$ is an edge in $K$.
Since $K$ is connected, there exists a path $x = v_0 \rightarrow v_1 \rightarrow \dots \rightarrow v_n = y$ in $K$, with $n$ odd.
If $n \geq 3$, then our assumption implies that $v_0 \rightarrow v_1 \rightarrow \dots \rightarrow v_{n-3} \rightarrow v_n$ is a path of length $n-2$ from $x$ to $y$.
Repeating this argument, we eventually get a path of length $1$ from $x$ to $y$, i.e., an edge $x \rightarrow y$.
\end{proof}

\begin{theorem}
\label{thm:main-undirected}
Let $G$ be an undirected graph.
\begin{enumerate}[label={\upshape{(\roman*)}}]
\item
If every connected component of $G$ is either trivial or a complete graph \textup{(}with loops\textup{)}, then $\GA{G}$ satisfies every bracketing identity.
In this case, $s_n(\GA{G}) = 1$ for all $n \in \IN_+$.

\item
If every connected component is either trivial, a complete graph \textup{(}with loops\textup{)}, or a complete bipartite graph, and the last case occurs at least once, then $\GA{G}$ satisfies a bracketing identity $t \approx t'$ if and only if $d_{G(t)}(x_i) \equiv d_{G(t')}(x_i) \pmod{2}$ for all $x_i \in X_n$.
In this case, $s_n(\GA{G}) = 2^{n-2}$ for all $n \geq 2$.

\item
Otherwise $\GA{G}$ satisfies no nontrivial bracketing identity.
In this case, $s_n(\GA{G}) = C_{n-1}$ for all $n \in \IN_+$.
\end{enumerate}
\end{theorem}

\begin{proof}
Let $t, t' \in B_n$, $t \neq t'$.
Denote $T := G(t)$, $T' := G(t')$.
Assume that $G = (V,E)$ satisfies $t \approx t'$.

\begin{clm}
\label{clm:complete}
Every connected component of $G$ containing a loop is a complete graph (with loops).
\end{clm}

\begin{proof}
Let $K$ be a connected component of $G$ containing a loop.
We will show that the edge relation $E(K)$ is reflexive, symmetric, and transitive.
From this we can conclude that $K$ is a complete graph with loops.
The claim is obvious if $K$ has only one vertex, so we may assume that $K$ has at least two vertices.
The edge relation is symmetric because $G$ is undirected.

For reflexivity, let $u$ be a vertex in $K$ with a loop, and let $v$ be a vertex adjacent to $u$.
It follows from Lemma~\ref{lem:loops}\ref{lem:loops:refl-symm} that $(v,v) \in E(K)$ (note that $v$ belongs to the cycle $v \rightarrow u \rightarrow v$).
From this we can conclude that every vertex in $K$ has a loop, that is, the edge relation $E(K)$ is reflexive.

For transitivity, assume that $(u,v)$ and $(v,w)$ are edges in $K$.
By reflexivity we have loops at vertices $u$, $v$ and $w$, and by symmetry we have also the edges $(v,u)$ and $(w,v)$.
Now Lemma~\ref{lem:loops}\ref{lem:loops:trans} implies that $(u,w) \in E(K)$.
\clmqed{clm:complete}
\end{proof}

\begin{clm}
\label{clm:bipartite}
Every nontrivial connected component of $G$ without loops is a complete bipartite graph.
Such a component exists only if
$d_T(x_i) \equiv d_{T'}(x_i) \pmod{2}$ for all $x_i \in X_n$.
\end{clm}

\begin{proof}
Let $K$ be a nontrivial connected component of $G$ without loops.
Then $K$ contains an edge $(u,v)$.
Then the map $\varphi \colon X_n \to V$,
\[
\varphi(x) =
\begin{cases}
u, & \text{if $d_T(x) \equiv 0 \pmod{2}$,} \\
v, & \text{if $d_T(x) \equiv 1 \pmod{2}$,}
\end{cases}
\]
is clearly a homomorphism of $T$ into $G$.
By Proposition~\ref{prop:phi}, $\varphi$ is a homomorphism of $T'$ into $G$, so $d_T(x_i) \equiv d_{T'}(x_i) \pmod{2}$ for all $x_i \in X_n$.

In order to conclude that $K$ is complete bipartite, it suffices, by Lemma~\ref{lem:complete-bipartite}, to show that if $a \rightarrow b \rightarrow c \rightarrow d$ is a walk in $K$, then $(a,d)$ is an edge.
Since the edge relation is symmetric, this holds by Lemma~\ref{lem:loops}\ref{lem:loops:even-shortcut}.
\clmqed{clm:bipartite}
\end{proof}

\textit{Proof of Theorem~\ref{thm:main-undirected} continued.}
Claims \ref{clm:complete} and \ref{clm:bipartite} show that if $G$ satisfies $t \approx t'$, then the connected components of $G$ are trivial, complete graphs (with loops), or complete bipartite graphs, and if the last case occurs, then $d_T(x_i) \equiv d_{T'}(x_i) \pmod{2}$ for all $x_i \in X_n$.

Assume now that the connected components of $G$ are trivial, complete graphs (with loops), or complete bipartite graphs, and if one of the components is a complete bipartite graph, then $d_T(x_i) \equiv d_{T'}(x_i) \pmod{2}$ for all $x_i \in X_n$.
In order to prove that $G$ satisfies $t \approx t'$, we apply Proposition~\ref{prop:phi}.
Let $\varphi \colon X_n \to V$ be a homomorphism of $T$ into $G$.
Since $T$ is connected and contains an edge, the image of $\varphi$ lies in a single nontrivial connected component $K$ of $G$.
If $K$ is a complete graph, then $\varphi$ is obviously a homomorphism of $T'$ into $G$.

Consider then the case where $K$ is a complete bipartite graph with bipartition $B_1, B_2$.
It is easy to see that for all $x_i, x_j \in X_n$, $\varphi(x_i)$ and $\varphi(x_j)$ lie in the same part ($B_1$ or $B_2$) if and only if $d_T(x_i) \equiv d_T(x_j) \pmod{2}$.
By our assumption, we have $d_T(x_i) \equiv d_{T'}(x_i) \pmod{2}$ for all $x_i \in X_n$,
which implies that $\varphi$ is also a homomorphism of $T'$ into $G$.

A similar argument shows that every homomorphism of $T'$ into $G$ is also a homomorphism of $T$ into $G$.
By Proposition~\ref{prop:phi}, $G$ satisfies $t \approx t'$.

We have shown that $G$ satisfies a nontrivial bracketing identity $t \approx t'$ if and only if 
the connected components of $G$ are trivial, complete graphs (with loops), or complete bipartite graphs,
and if the last case occurs, then $d_T(x_i) \equiv d_{T'}(x_i) \pmod{2}$ for all $x_i \in X_n$.
This gives us the three possible associative spectra.
If $\GA{G}$ satisfies all bracketing identities, then $\GA{G}$ is associative and $s_n(\GA{G}) = 1$ for all $n \in \IN_+$.
If $\GA{G}$ satisfies no nontrivial bracketing identity, then $s_n(\GA{G}) = C_{n-1}$ for all $n \in \IN_+$.
In the last possible case, $\sigma_n(\GA{G})$ relates $t$ and $t'$ if and only if $d_T(x_i) \equiv d_{T'}(x_i) \pmod{2}$ for all $x_i \in X_n$; by Lemma~\ref{lem:H2} we have $s_n(\GA{G}) = 2^{n-2}$.
\end{proof}

\section{Associative and antiassociative digraphs}
\label{sect directed}

A digraph $G$ is \emph{associative} if $\GA{G}$ satisfies the associative identity 
$x_1 (x_2 x_3) \approx (x_1 x_2) x_3$, i.e., if the associative spectrum of 
$\GA{G}$ is constant $1$.
Associative digraphs were characterized by T.~Poomsa\hyp{}ard~\cite{Poo-2000};
the equivalence of conditions \ref{prop:associative:associative} and \ref{prop:associative:edges} in the following can be verified by applying Proposition~\ref{prop:phi} to the DFS trees associated with the two bracketings appearing in the associative law (see Figure~\ref{fig:associative}).

\begin{figure}
\tikzset{every node/.style={circle,draw,inner sep=1.5,fill=black}, every path/.style={->,>=stealth}}
\begin{tikzpicture}[scale=1, baseline=(x1)]
\node[label=below:{$x_1$}] (x1) at (0,0) {};
\node[label=left:{$x_2$}] (x2) [above of=x1] {};
\node[label=left:{$x_3$},label=right:{\phantom{$x_3$}}] (x3) [above of=x2] {};
\node[rectangle,draw=none,fill=none] () [below of=x1] {$G(t)$};
\path (x1) edge (x2) edge (x3);
\end{tikzpicture}
\qquad
\begin{tikzpicture}[scale=1, baseline=(x1)]
\node[label=below:{$x_1$}] (x1) at (0,0) {};
\node[label=left:{$x_2$}] (x2) [above left of=x1] {};
\node[label=right:{$x_3$}] (x3) [above right of=x1] {};
\node[rectangle,draw=none,fill=none] () [below of=x1] {$G(t')$};
\path (x1) edge (x2) edge (x3);
\end{tikzpicture}
\caption{Graphs associated with the terms of the associative identity $t \approx t'$ with $t := x_1 (x_2 x_3)$, $t' := (x_1 x_2) x_3$.}
\label{fig:associative}
\end{figure}
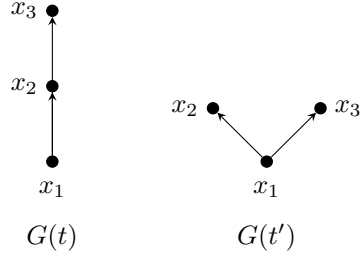

\begin{proposition}[{Poomsa\hyp{}ard~\cite[Proposition~2.2]{Poo-2000}}]
\label{prop:associative}
For any digraph $G = (V,E)$, the following statements are equivalent.
\begin{enumerate}[label={\upshape{(\roman*)}}]
\item\label{prop:associative:associative} $G$ is associative.
\item\label{prop:associative:edges} For any edge $(u,v) \in E$ and and any vertex $w \in V$, $(u,w) \in E$ if and only if $(v,w) \in E$.
\item The edge relation $E$ is transitive and for every $v \in V$, the subgraph induced by $\Nout[G]{v}$ is a complete graph.
\end{enumerate}
\end{proposition}

On the other extreme, we have the \emph{antiassociative} digraphs whose graph algebras
satisfy no nontrivial bracketing identities, i.e., the associative spectrum of 
$\GA{G}$ consists of the Catalan numbers.
The goal of this section is to characterize antiassociative digraphs.
To this end, we introduce some numerical parameters of bracketing identities
in terms of the corresponding DFS trees, and we prove several necessary conditions
for a graph algebra to satisfy a given bracketing identity. 

\begin{definition}
\label{def:H}\label{def:M}\label{def:L}
Let $t, t' \in B_n$, $t \neq t'$, and let $T := G(t)$, $T' := G(t')$.
\begin{enumerate}[label={\upshape{(\roman*)}}]
\item Let $\Htt := \min \{h(T), h(T')\}$.
\item Let $\Mtt$ be the largest integer $m$ such that $d_T(x_i) \equiv d_{T'}(x_i) \pmod{m}$ for all $x_i \in X_n$.
In other words, the depth sequences of $T$ and $T'$ are congruent modulo $\Mtt$.
\item Let $\Ltt$ be the largest integer $m$ such that for all $x_i \in X_n$,
\[
\bigl( d_T(x_i) \leq m \vee d_{T'}(x_i) \leq m \bigr) \implies d_T(x_i) = d_{T'}(x_i).
\]
In other words, the DFS trees $T$ and $T'$ are identical up to level $\Ltt$.
\end{enumerate}
Note that $0 \leq \Htt < n$ (with $\Htt = 0$ if and only if $n = 1$), $0 \leq \Ltt < \Htt$ and $1 \leq \Mtt \leq \Htt$.
\end{definition}

\begin{example}
\label{ex:HML}
Figure~\ref{fig:ex:HML} shows two DFS trees corresponding to certain terms $t, t' \in B_{20}$.
It is straightforward to verify that $\Htt = 6$, $\Mtt = 3$, and $\Ltt = 2$.
\end{example}
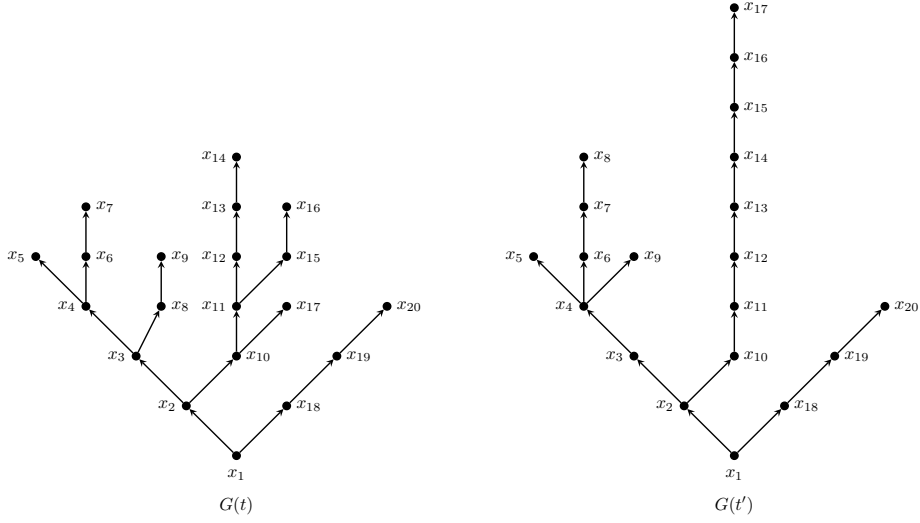
\begin{figure}
\tikzset{every node/.style={circle,draw,inner sep=1.5,fill=black}, every path/.style={->,>=stealth,thick}}
\scalebox{0.66}{
\begin{tikzpicture}[scale=1, baseline=(x1)]
\node[label=below:{$x_1$}] (x1) at (0,0) {};
\node[label=left:{$x_2$}] (x2) at (-1,1) {};
\node[label=left:{$x_3$}] (x3) at (-2,2) {};
\node[label=left:{$x_4$}] (x4) at (-3,3) {};
\node[label=left:{$x_5$}] (x5) at (-4,4) {};
\node[label=right:{$x_6$}] (x6) at (-3,4) {};
\node[label=right:{$x_7$}] (x7) at (-3,5) {};
\node[label=right:{$x_8$}] (x8) at (-1.5,3) {};
\node[label=right:{$x_9$}] (x9) at (-1.5,4) {};
\node[label=right:{$x_{10}$}] (x10) at (0,2) {};
\node[label=left:{$x_{11}$}] (x11) at (0,3) {};
\node[label=left:{$x_{12}$}] (x12) at (0,4) {};
\node[label=left:{$x_{13}$}] (x13) at (0,5) {};
\node[label=left:{$x_{14}$}] (x14) at (0,6) {};
\node[label=right:{$x_{15}$}] (x15) at (1,4) {};
\node[label=right:{$x_{16}$}] (x16) at (1,5) {};
\node[label=right:{$x_{17}$}] (x17) at (1,3) {};
\node[label=right:{$x_{18}$}] (x18) at (1,1) {};
\node[label=right:{$x_{19}$}] (x19) at (2,2) {};
\node[label=right:{$x_{20}$}] (x20) at (3,3) {};
\node[rectangle,draw=none,fill=none] () [below of=x1] {$G(t)$};
\path (x1) edge (x2);
\path (x2) edge (x3);
\path (x3) edge (x4);
\path (x4) edge (x5);
\path (x4) edge (x6);
\path (x6) edge (x7);
\path (x3) edge (x8);
\path (x8) edge (x9);
\path (x2) edge (x10);
\path (x10) edge (x11);
\path (x11) edge (x12);
\path (x12) edge (x13);
\path (x13) edge (x14);
\path (x11) edge (x15);
\path (x15) edge (x16);
\path (x10) edge (x17);
\path (x1) edge (x18);
\path (x18) edge (x19);
\path (x19) edge (x20);
\end{tikzpicture}
}
\qquad
\scalebox{0.66}{
\begin{tikzpicture}[scale=1, baseline=(x1)]
\node[label=below:{$x_1$}] (x1) at (0,0) {};
\node[label=left:{$x_2$}] (x2) at (-1,1) {};
\node[label=left:{$x_3$}] (x3) at (-2,2) {};
\node[label=left:{$x_4$}] (x4) at (-3,3) {};
\node[label=left:{$x_5$}] (x5) at (-4,4) {};
\node[label=right:{$x_6$}] (x6) at (-3,4) {};
\node[label=right:{$x_7$}] (x7) at (-3,5) {};
\node[label=right:{$x_8$}] (x8) at (-3,6) {};
\node[label=right:{$x_9$}] (x9) at (-2,4) {};
\node[label=right:{$x_{10}$}] (x10) at (0,2) {};
\node[label=right:{$x_{11}$}] (x11) at (0,3) {};
\node[label=right:{$x_{12}$}] (x12) at (0,4) {};
\node[label=right:{$x_{13}$}] (x13) at (0,5) {};
\node[label=right:{$x_{14}$}] (x14) at (0,6) {};
\node[label=right:{$x_{15}$}] (x15) at (0,7) {};
\node[label=right:{$x_{16}$}] (x16) at (0,8) {};
\node[label=right:{$x_{17}$}] (x17) at (0,9) {};
\node[label=right:{$x_{18}$}] (x18) at (1,1) {};
\node[label=right:{$x_{19}$}] (x19) at (2,2) {};
\node[label=right:{$x_{20}$}] (x20) at (3,3) {};
\node[rectangle,draw=none,fill=none] () [below of=x1] {$G(t')$};
\path (x1) edge (x2);
\path (x2) edge (x3);
\path (x3) edge (x4);
\path (x4) edge (x5);
\path (x4) edge (x6);
\path (x6) edge (x7);
\path (x7) edge (x8);
\path (x4) edge (x9);
\path (x2) edge (x10);
\path (x10) edge (x11);
\path (x11) edge (x12);
\path (x12) edge (x13);
\path (x13) edge (x14);
\path (x14) edge (x15);
\path (x15) edge (x16);
\path (x16) edge (x17);
\path (x1) edge (x18);
\path (x18) edge (x19);
\path (x19) edge (x20);
\end{tikzpicture}
}
\caption{DFS trees with $\Htt = 6$, $\Mtt = 3$, $\Ltt = 2$.}
\label{fig:ex:HML}
\end{figure}

\begin{lemma}
\label{lem:edge-to-L+1}
Let $t, t' \in B_n$, $t \neq t'$, and
let $G$ be a digraph such that $\GA{G}$ satisfies the identity $t \approx t'$.
Denote $H := \Htt$, $M := \Mtt$, $L := \Ltt$.
Then there exists an integer $r$ with $L + 1 \leq r \leq H$ and $r \equiv L \pmod{M}$ such that the following holds:
if $v_0 \rightarrow v_1 \rightarrow \dots \rightarrow v_H$ is a walk in $G$, then $v_r \rightarrow v_{L+1}$ is an edge in $G$.
In particular, $v_{L+1}$ belongs to a nontrivial strongly connected component.
\end{lemma}

\begin{proof}
By the definition of $L$, there exists a vertex $x_d \in X_n$ such that either $d_T(x_d) = L + 1 < d_{T'}(x_d)$ or $d_{T'}(x_d) = L + 1 < d_T(x_d)$.
By changing the roles of $T$ and $T'$, if necessary, we may assume that $d_T(x_d) = L + 1 < d_{T'}(x_d)$.
Let $x_p$ be the parent of $x_d$ in $T$, and let $x_q$ be the parent of $x_d$ in $T'$.
Then $d_T(x_p) = L$, and it follows from
Definition~\ref{def:DFS}
that $V(T_{x_p}) = V(T'_{x_p})$ because the trees $T$ and $T'$ are identical up to level $L$.
Since $x_d \in V(T_{x_p}) = V(T'_{x_p})$ and $d_{T'}(x_d) > L + 1$, we have $x_q \in V(T'_{x_p}) = V(T_{x_p})$ and $d_{T'}(x_q) = d_{T'}(x_d) - 1 \geq L+1$, so $x_q \neq x_p$; hence $d_T(x_q) \geq L + 1$.
Furthermore, by the definition of $M$, it holds that
$d_T(x_q) \equiv d_{T'}(x_q) = d_{T'}(x_d) - 1 \equiv d_T(x_d) - 1 = L \pmod{M}$.

Write $h := h(T)$, $h' := h(T')$, and consider first the case that $h \leq h'$, so $h = H$.
In this case, the statement holds with $r := d_T(x_q)$, because $L+1 \leq r = d_T(x_q) \leq h = H$ and $r = d_T(x_q) \equiv L \pmod{M}$, and
by Lemma~\ref{lem:TT'edges}, it holds that if $v_0 \rightarrow v_1 \rightarrow \dots \rightarrow v_H$ is a walk in $G$, then
$(v_{d_T(x_q)}, v_{d_T(x_d)}) = (v_r, v_{L+1}) \in E(G)$.

Consider now the case that $h > h'$, so $h' = H$.
Let $u_0 \rightarrow u_1 \rightarrow \dots \rightarrow u_h$ be a longest path in $T$, and write $d_i := d_{T'}(u_i)$ for $i \in \{0, \dots, h\}$.
Now, since $h > h'$, the sequence $d_0, d_1, \dots, d_h$ cannot be strictly increasing, so there exists an index $j$ with $d_j \geq d_{j+1}$.
Note that $d_{j+1} \geq L+1$, because the trees $T$ and $T'$ are identical up to level $L$.

Assume that $W \colon v_0 \rightarrow v_1 \rightarrow \dots \rightarrow v_H$ is a walk in $G$.
By Lemma~\ref{lem:TT'edges}, $(v_{d_j}, v_{d_{j+1}}) \in E(G)$; consequently
$C \colon v_{d_{j+1}} \rightarrow v_{d_{j+1} + 1} \rightarrow \dots \rightarrow v_{d_j} \rightarrow v_{d_{j+1}}$ is a closed walk in $G$.
Now, let $W'$ be the walk in $G$ that starts with $v_0 \rightarrow v_1 \rightarrow \dots \rightarrow v_{d_{j+1}}$, and then it continues around the closed walk $C$ until it reaches length $h$.
More precisely, $W'$ is the walk $v'_0 \rightarrow v'_1 \rightarrow \dots \rightarrow v'_h$
with
$v'_i := v_{i^*}$, where
$i^*$ is the largest integer $m$ such that $m \leq \min(i, d_j)$ and $m \equiv i \pmod{D}$, where $D := d_j - d_{j+1} + 1$.
By Lemma~\ref{lem:TT'edges}, $(v'_{d_T(x_q)}, v'_{d_T(x_d)}) = (v_r, v_{L+1}) \in E(G)$, where
$r := (d_T(x_q))^*$.
By definition, we have $L + 1 \leq r \leq d_j \leq h' = H$
and $r \equiv d_T(x_q) \pmod{D}$.
Furthermore,
\[
D = d_{T'}(u_j) - d_{T'}(u_{j+1}) + 1
\equiv d_T(u_j) - d_T(u_{j+1}) + 1
= 0 \pmod{M},
\]
so $M \divides D$, and it follows that
$r \equiv d_T(x_q) \equiv L \pmod{M}$.

Now we have a closed walk $v_{L+1} \rightarrow \dots \rightarrow v_r \rightarrow v_{L+1}$ in $G$.
This means, in particular, that $v_{L+1}$ belongs to a nontrivial strongly connected component.
\end{proof}

The next lemma generalizes Lemma~\ref{lem:loops}\ref{lem:loops:refl-symm}.

\begin{lemma}
\label{lem:m-cycle-out}
Let $t, t' \in B_n$, $t \neq t'$, and
let $G$ be a digraph such that $\GA{G}$ satisfies the identity $t \approx t'$.
Denote $H := \Htt$, $M := \Mtt$, $L := \Ltt$.
If $m$ is a divisor of $M$,
$U \colon u_0 \rightarrow u_1 \rightarrow \dots \rightarrow u_{m-1} \rightarrow u_0$ is a closed walk in $G$, $u_0 \rightarrow w$ is an edge, and $G$ contains arbitrarily long walks with initial vertex $w$, then $w \rightarrow u_2$ if $m > 2$ and $w \rightarrow u_0$ if $1 \leq m \leq 2$.
\end{lemma}

\begin{proof}
For $m = 1$ this is Lemma~\ref{lem:loops}\ref{lem:loops:refl-symm}.
Assume that $m \geq 2$.
Let $r$ be the number provided by
Lemma~\ref{lem:edge-to-L+1},
and let $v_0 \rightarrow v_1 \rightarrow \dots \rightarrow v_H$ be a walk that starts by going around the closed walk $U$ so that $v_{r-1} = u_0$
(i.e., $v_i := u_{i - r + 1 \bmod m}$ for $0 \leq i \leq r-1$) and continues with a walk of length $H-r$ with initial vertex $w$.
By
Lemma~\ref{lem:edge-to-L+1},
$v_r \rightarrow v_{L+1}$ is an edge.
We have $v_r = w$ and $v_{L+1} = u_{L + 1 - r + 1 \bmod m}$.
Since $r \equiv L \pmod{M}$ and $m \divides M$, we have $r \equiv L \pmod{m}$;
hence $L + 1 - r + 1 \equiv 2 \mod{m}$.
Therefore $w \rightarrow u_2$ is an edge if $m > 2$ and $w \rightarrow u_0$ is an edge if $m = 2$.
\end{proof}

\begin{lemma}
\label{lem:m-cycle}
Let $t, t' \in B_n$, $t \neq t'$.
Then for every $m \in \IN_+$, the directed $m$\hyp{}cycle $C_m$ satisfies $t \approx t'$ if and only if $m$ is a divisor of $\Mtt$.
\end{lemma}

\begin{proof}
Denote $M := \Mtt$.
For notational simplicity, we suppose that $V(C_m) = \ZZ_m$ and
for all $i, j \in \ZZ_m$, $(i,j) \in E(C_m)$ if and only if $j \equiv i + 1 \pmod{m}$.

Assume first that $m \divides M$.
By the definition of $M$, we have $d_T(x_i) \equiv d_{T'}(x_i) \pmod{M}$ for all $x_i \in X_n$.
Since $m \divides M$, this implies $d_T(x_i) \equiv d_{T'}(x_i) \pmod{m}$ for all $x_i \in X_n$.
Let $\varphi \colon T \to C_m$ be a homomorphism.
Then $\varphi$ is necessarily of the form $x_i \mapsto (d_T(x_i) + k) \bmod{m}$ for some fixed $k \in \ZZ_m$ (that is, $\varphi$ collapses $T$ onto $C_m$).
Then for every edge $(x_i, x_j)$ of $T'$, we have
\begin{multline*}
\varphi(x_j)
\equiv d_T(x_j) + k
\equiv d_{T'}(x_j) + k
= d_{T'}(x_i) + 1 + k
\\
\equiv d_T(x_i) + k + 1
\equiv \varphi(x_i) + 1
\pmod{m},
\end{multline*}
so $(\varphi(x_i), \varphi(x_j))$ is an edge of $C_m$.
Therefore $\varphi$ is a homomorphism of $T'$ into $C_m$.
A similar argument shows that every homomorphism $\varphi \colon T' \to C_m$ is a homomorphism $T \to C_m$.
By Proposition~\ref{prop:phi}, $C_m$ satisfies $t \approx t'$.

Assume now that $C_m$ satisfies $t \approx t'$.
Let $\varphi \colon T \to C_m$ be the collapsing map of $T$ on $C_m$ with $\varphi(x_i) = d_T(x_i) \bmod{m}$.
By Proposition~\ref{prop:phi}, $\varphi$ is a homomorphism $T' \to C_m$.
Since the only homomorphisms of $T'$ to $C_m$ are collapsing maps $x_i \mapsto (d_{T'}(x_i) + k) \bmod m$ for some $k \in \ZZ_m$, and since $\varphi(x_1) = d_T(x_1) = 0 = d_{T'}(x_1)$, it follows that $d_{T'}(x_i) \equiv d_T(x_i) \pmod{m}$ for all $x_i \in X_n$.
From the definition of $M$ it follows that $m \divides M$.
\end{proof}

\begin{definition}
\label{def:MG}
A digraph $G = (V,E)$ is called an \emph{$m$\hyp{}whirl} ($m \in \IN_+$), if there exists a partition $\{B_0, \dots, B_{m-1}\}$ of $V$ such that for all $x, y \in V$, $(x,y) \in E$ if and only if $x \in B_i$ and $y \in B_{i+1}$ for some $i \in \{0, \dots, m-1\}$ (addition modulo $m$).
The sets $B_i$ are referred to as the \emph{blocks} of $G$.
We say that blocks $B_i$ and $B_j$ are \emph{consecutive} if $j \equiv i + 1 \pmod{m}$;
then $B_i$ is called the \emph{predecessor} of $B_j$ and $B_j$ is called the \emph{successor} of $B_i$.
A digraph is called a \emph{whirl} if it is an $m$\hyp{}whirl for some $m \in \IN_+$.

In other words, an $m$\hyp{}whirl $G$ is a strong homomorphic preimage of the directed $m$\hyp{}cycle $C_m$.
By definition, $1$\hyp{}whirls are precisely the complete graphs with loops,
and $2$\hyp{}whirls are precisely the complete bipartite graphs.
(Note the role of $1$\hyp{} and $2$\hyp{}whirls in Theorem~\ref{thm:main-undirected}.)
\end{definition}

\begin{lemma}
\label{lem:M}
Let $t, t' \in B_n$, $t \neq t'$, and
let $G$ be a digraph such that $\GA{G}$ satisfies the identity $t \approx t'$.
Then every strongly connected component of $G$ is either trivial or
an $m$\hyp{}whirl for some divisor $m$ of $\Mtt$.
\end{lemma}

\begin{proof}
Let $K$ be a nontrivial strongly connected component of $G$.
Every vertex of $K$ lies on a cycle contained in $K$;
let $C$ be a shortest cycle in $K$, and assume that $C$ has length $m$.
Lemma~\ref{lem:m-cycle} implies that $m$ is a divisor of $\Mtt$.
We want to show that $K$ is an $m$\hyp{}whirl.

Assume first that $m = 1$.
Then $K$ contains a vertex $u$ with a loop.
Let $v$ be an arbitrary vertex of $K$.
By strong connectivity, there exists a path from $u$ to $v$.
Since every vertex is contained in a cycle (again by strong connectivity), we can deduce with the help of Lemma~\ref{lem:loops}\ref{lem:loops:refl-symm} that every vertex along the path from $u$ to $v$ has a loop.
We conclude that the edge relation of $K$ is reflexive.
The reflexivity of the edge relation and Lemma~\ref{lem:loops}\ref{lem:loops:refl-symm} imply immediately that the edge relation is symmetric.
Now Lemma~\ref{lem:loops}\ref{lem:loops:trans} implies in turn that the edge relation is transitive.
We conclude that $K$ is a complete graph with loops, i.e., a $1$\hyp{}whirl.

Assume now that $m = 2$.
Then $\Mtt$ is even, $K$ contains no loop, and there is a cycle of length $2$ in $K$, i.e., vertices $u$, $v$ with $(u,v), (v,u) \in E(G)$.
With the help of strong connectivity and Lemma~\ref{lem:loops}\ref{lem:loops:even-symm}, we can deduce that $K$ is undirected.
It now follows from Lemmas~\ref{lem:complete-bipartite} and \ref{lem:loops}\ref{lem:loops:even-shortcut} that $K$ is a complete bipartite graph, i.e., a $2$\hyp{}whirl.

From now on, assume that $m > 2$.
For notational simplicity, suppose that $C = \ZZ_m$
and for all $i, j \in \ZZ_m$, $i \rightarrow j$ is an edge if and only if $j \equiv i + 1 \pmod{m}$.
For each $i \in \ZZ_m$, let
\[
B_i :=
\Nout[K]{i-1}
= \Nout[G]{i-1} \cap V(K)
= \{v \in V(K) \mid (i-1, v) \in E(K)\},
\]
i.e., $B_i$ is the set of all outneighbours of $i-1$ (addition modulo $m$) belonging to the strongly connected component $K$.
Note that $i \in B_i$ by definition.
We show that for all $i \in \ZZ_m$ and for all $v \in B_i$, we have
$\Nout[K]{v} = B_{i+1}$.
Let $i \in \ZZ_m$ and $v \in B_i$.
Considering
the closed walk
$i-1 \rightarrow i \rightarrow \dots \rightarrow m-1 \rightarrow 0 \rightarrow \dots \rightarrow i-1$ of length $m$
and the edge $i-1 \rightarrow v$,
Lemma~\ref{lem:m-cycle-out}
gives the edge $v \rightarrow i+1$
(note that by strong connectivity every vertex of $K$, in particular $v$, belongs to a cycle).
Now let $i \in \ZZ_m$, $v \in B_i$ and $w \in B_{i+1}$.
Considering the closed walk
$i-1 \rightarrow i \rightarrow w \rightarrow i+2 \rightarrow \dots \rightarrow m-1 \rightarrow 0 \rightarrow \dots \rightarrow i-1$ of length $m$
and the edge $i-1 \rightarrow v$,
Lemma~\ref{lem:m-cycle-out}
gives the edge $v \rightarrow w$.
We have shown thus far that for all $i \in \ZZ_m$,
$B_{i+1} \subseteq \Nout[K]{v}$ for all $v \in B_i$.

Now let $i \in \ZZ_m$, $v \in B_i$, and let $w$ be a vertex of $K$ with $v \rightarrow w$.
We have shown above that $v \rightarrow i+1$ is an edge.
Considering the closed walk
$v \rightarrow i+1 \rightarrow \dots \rightarrow m-1 \rightarrow 0 \rightarrow 1 \rightarrow \dots \rightarrow i-1 \rightarrow v$ of length $m$
and the edge $v \rightarrow w$,
Lemma~\ref{lem:m-cycle-out}
gives the edge $w \rightarrow i+2$.
Considering the closed walk
$i-1 \rightarrow v \rightarrow w \rightarrow i+2 \rightarrow \dots \rightarrow m-1 \rightarrow 0 \rightarrow 1 \rightarrow \dots \rightarrow i-1$ of length $m$
and the edge $i-1 \rightarrow i$,
Lemma~\ref{lem:m-cycle-out}
gives the edge $i \rightarrow w$.
Thus $w \in B_{i+1}$.
This shows that for each vertex $v$ of $B_i$, $\Nout[K]{v} \subseteq B_{i+1}$.

It remains to show that the sets $B_0, B_1, \dots, B_{m-1}$ constitute a partition of $V(K)$.
Let us show first that these sets are pairwise disjoint.
Suppose, to the contrary, that $B_i \cap B_j \neq \emptyset$ for some $i \neq j$, and let $v \in B_i \cap B_j$.
Then we have $i-1 \rightarrow v \rightarrow i+1$ and $j-1 \rightarrow v \rightarrow j+1$.
We will find a contradiction by showing that $K$ contains a cycle shorter than $C$.
If $j = i+1$,
then $K$ contains the loop $v \rightarrow v$, a cycle of length $1$.
Otherwise $v \rightarrow i+1 \rightarrow \dots \rightarrow j-1 \rightarrow v$ is a cycle shorter than $C$.

Suppose now, to the contrary, that $\bigcup_{i=0}^{m-1} B_i \neq V(K)$, and let $v \in V(K) \setminus \bigcup_{i=0}^{m-1} B_i$.
Since $K$ is strongly connected, there exists a path $0 = v_0 \rightarrow v_1 \rightarrow \dots \rightarrow v_p = v$ in $K$.
Then there exists an index $q \in \{0, \dots, p-1\}$ such that $v_q \in \bigcup_{i=0}^{m-1} B_i$ and $v_{q+1} \notin \bigcup_{i=0}^{m-1} B_i$, say $v_q \in B_j$.
But we have shown above that 
$\Nout[K]{v_q} = B_{j+1}$.
This gives the desired contradiction, and
we conclude that $K$ is an $m$\hyp{}whirl.
\end{proof}

\begin{lemma}
\label{lem:no-SCC-SCC}
Let $t, t' \in B_n$, $t \neq t'$, and
let $G$ be a digraph such that $\GA{G}$ satisfies the identity $t \approx t'$.
Then there is no path from a nontrivial strongly connected component of $G$ to another.
\end{lemma}

\begin{proof}
Suppose, to the contrary, that there are distinct nontrivial strongly connected components $K$ and $K'$ and a path $P$ from a vertex $v \in V(K)$ to a vertex $v' \in V(K')$.
Assume that $P$ is the shortest possible among all such paths.
Then $v$ is the only vertex of $P$ lying in $K$.
Let $w$ be the successor of $v$ along this path.

By Lemma~\ref{lem:M}, $K$ and $K'$ are $m$\hyp{} and $m'$\hyp{}whirls, respectively, for some divisors $m$ and $m'$ of $\Mtt$.
Hence $v$ belongs to an $m$\hyp{}cycle $C$ in $K$ and $v'$ belongs to an $m'$\hyp{}cycle $C'$ in $K'$.
Now Lemma~\ref{lem:m-cycle-out} provides an edge from $w$ to a vertex on $C$.
This means that $w$ belongs to the strongly connected component $K$, a contradiction.
\end{proof}

\begin{definition}
\label{def:PG}\label{def:EG}\label{def:OG}
Let $G = (V,E)$ be a digraph.
A walk in $G$ is \emph{pleasant,} if all its vertices belong to trivial strongly connected components.
Every pleasant walk is a path.
\end{definition}

\begin{lemma}
\label{lem:H}
Let $t, t' \in B_n$, $t \neq t'$, and
let $G$ be a digraph such that $\GA{G}$ satisfies the identity $t \approx t'$.
Then $G$ has no pleasant path of length $\Htt$.
\end{lemma}

\begin{proof}
Denote $H := \Htt$, $L := \Ltt$.
Suppose, to the contrary, that there exists a pleasant path $v_0 \rightarrow v_1 \rightarrow \dots \rightarrow v_H$ in $G$.
By Lemma~\ref{lem:edge-to-L+1}, there is an index $r$ with $L+1 \leq r \leq H$ such that $(v_r, v_{L+1})$ is an edge of $G$.
Consequently, $v_{L+1}$ belongs to a nontrivial strongly connected component, a contradiction.
\end{proof}

The necessary conditions established in the previous lemmas together are sufficient for
the satisfaction of \emph{some} nontrivial bracketing identity.
We prove this in the following theorem, which provides a complete characterization of
(not) antiassociative digraphs. 
Note that this does not constitute a
necessary and sufficient condition for the satisfaction of a \emph{given} nontrivial bracketing identity. Finding such a condition will
be a topic of a forthcoming paper.

\begin{theorem}
\label{thm:antiass}
Let $G$ be a digraph.
Then $\GA{G}$ is not antiassociative if and only if the following conditions hold.
\begin{enumerate}[label={\upshape{(\roman*)}}]
\item\label{prop:antiass:whirl} Every nontrivial strongly connected component of $G$ is a whirl.
\item\label{prop:antiass:path} There is no path from a nontrivial strongly connected component of $G$ to another.
\item\label{prop:antiass:H} There is a finite upper bound on the length of the pleasant paths in $G$.
\item\label{prop:antiass:M} There is a finite upper bound on the numbers $m$ such that $G$ contains an $m$\hyp{}whirl.
\end{enumerate}
\end{theorem}

\begin{proof}
Assume that $\GA{G}$ satisfies a nontrivial bracketing identity $t \approx t'$.
By
Lemma~\ref{lem:edge-to-L+1},
the pleasant paths in $G$ have length less than $\Htt$.
By Lemma~\ref{lem:M}, every nontrivial strongly connected component of $G$ is an $m$\hyp{}whirl for some divisor $m$ of $\Mtt$; such numbers $m$ are clearly bounded above by $\Mtt$.
By Lemma~\ref{lem:no-SCC-SCC}, there is no path from a nontrivial strongly connected component of $G$ to another.

Assume now that conditions \ref{prop:antiass:whirl}--\ref{prop:antiass:M} hold.
We will construct a nontrivial bracketing identity $t \approx t'$ that is satisfied by $\GA{G}$.
We will define the terms $t$ and $t'$ in terms of the corresponding DFS trees $T := G(t)$ and $T' := G(t')$.
Let $P$ be an upper bound on the lengths of pleasant paths in $G$, as provided by condition \ref{prop:antiass:H},
and let
\[
M := \lcm \{ m \in \IN_{+} \mid \text{$G$ contains an $m$\hyp{}whirl} \},
\]
which is a finite natural number by condition \ref{prop:antiass:M},
with the convention that $\lcm \emptyset = 1$.
Let $n := 3 P + M + 6$, and let $T$ consist of the paths $x_1 \rightarrow \dots \rightarrow x_{2 P + M + 4}$ and $x_{2 P + M + 5} \rightarrow \dots \rightarrow x_{3 P + M + 6}$ and of the edge $x_{P + 2} \rightarrow x_{2 P + M + 5}$.
The tree $T'$ is constructed in a similar way, but we replace the edge $x_{P + 2} \rightarrow x_{2 P + M + 5}$ by $x_{P + M + 2} \rightarrow x_{2 P + M + 5}$.
If $\varphi \colon X_n \to V(G)$ is a homomorphism of $T$ into $G$, then there is an $i \in [1, P + 2]$ such that $\varphi(x_i)$ belongs to a nontrivial strongly connected component, by the definition of $P$.
Similarly, $\varphi(x_j)$ and $\varphi(x_k)$ belong to a nontrivial strongly connected component for some $j \in [P + M + 3, 2 P + M + 4]$ and for some $k \in [2 P + M + 5, 3 P + M + 6]$.
Condition \ref{prop:antiass:path} implies that $\varphi(x_i)$, $\varphi(x_{i+1})$, \dots, $\varphi(x_j)$ are in the same nontrivial strongly connected component $K$, and this includes the vertices $\varphi(x_{P+2})$, $\varphi(x_{P+3})$, \dots, $\varphi(x_{P + M + 3})$.
Similarly, also the vertices $\varphi(x_{2 P + M + 5})$, \dots, $\varphi(x_k)$ belong to $K$.
By the definition of $M$, the component $K$ is an $m$\hyp{}whirl for some divisor $m$ of $M$.
Therefore the vertices $\varphi(x_{P+2})$ and $\varphi(x_{P+M+2})$ belong to the same block $B$ of $K$,
and the vertices $\varphi(x_{P+3})$, $\varphi(x_{P+M+3})$ and $\varphi(x_{2 P + M + 5})$ belong to the successor block $B'$ of $B$.
This implies that $\varphi(x_{P + M + 2}) \rightarrow \varphi(x_{2 P + M + 5})$ is an edge, which proves that $\varphi$ is also a homomorphism of $T'$ into $G$.
An analogous argument shows that if $\varphi$ is a homomorphism of $T'$ into $G$, then $\varphi$ is also a homomorphism of $T$ into $G$.
Now if we let $t$ and $t'$ be the bracketings corresponding to $T$ and $T'$, respectively, then $\GA{G}$ satisfies $t \approx t'$ by Proposition~\ref{prop:phi}.
\end{proof}

\section{Some examples}
\label{sect examples}

In this section we determine the associative spectrum of a few special directed graphs
that are not covered by the results of the previous sections (i.e., they are 
neither associative nor antiassociative), such as directed paths and cycles, 
all graphs on two vertices, etc. For the spectra of directed paths,
we shall need the number of DFS trees of bounded height, so first we recall some
known facts about these numbers.

Let us denote the number of DFS trees of size $n$ of height at most $h$ by $\thn{h}{n}$.
The generating function $\sum_{n=0}^{\infty} \thn{h}{n} \cdot x^n$ is a rational function
(see \cite{BKR}), hence the sequence $\{ \thn{h}{n} \}_{n=0}^{\infty}$ satisfies a 
linear recurrence relation:
\begin{align*}
\thn{h}{n+1}
&=\binom{h}{1}\thn{h}{n}-\binom{h-1}{2}\thn{h}{n-1}+\binom{h-2}{3}\thn{h}{n-2}- \dots \\
&=\sum_{k=0}^{\left\lfloor\frac{h-1}{2}\right\rfloor} (-1)^k \binom{h-k}{k+1}\thn{h}{n-k}.
\end{align*}
We list these recurrence relations for $h=2,3,4,5$ in Table~\ref{table:thn} together with
explicit formulas for $\thn{h}{n}$ and the corresponding OEIS entries
(for $h=5$ the characteristic polynomial
of the linear recurrence is $x^3-5x^2+6x-1$, and its roots are not ``nice",
so we do not give an explicit formula for this case). Note that we have every second Fibonacci number for
$h=3$ (we use the indexing $F_1=F_2=1$).
For more information on the numbers $\thn{h}{n}$, see \cite{BKR}
(note that in \cite{BKR} the height is defined as the number of vertices of the longest path 
starting at the root, whereas in this paper the number of edges is counted), 
and see also the OEIS entry \href{https://oeis.org/A080934}{A080934}

\begin{table}
\begin{center}
\begin{tabular}{lll}
\toprule
recurrence & explicit formula & OEIS entry \\
\midrule
$\thn{2}{n+1}=2\thn{2}{n}$   &   $\thn{2}{n}=2^{n-2}$ & \href{https://oeis.org/A000079}{A000079} \\
$\thn{3}{n+1}=3\thn{3}{n}-\thn{3}{n-1}$   &   $\thn{3}{n}=F_{2n-3}$ & \href{https://oeis.org/A001519}{A001519} \\
$\thn{4}{n+1}=4\thn{4}{n}-3\thn{4}{n-1}$   &   $\thn{4}{n}=\frac{3^{n-2}+1}{2}$ & \href{https://oeis.org/A007051}{A007051} \\
$\thn{5}{n+1}=5\thn{5}{n}-6\thn{5}{n-1}+\thn{5}{n-2}$ & & \href{https://oeis.org/A080937}{A080937} \\
\bottomrule
\end{tabular}
\end{center}

\medskip
\caption{Number of DFS trees with bounded height.}
\label{table:thn}
\end{table}

\begin{proposition}
\label{prop:spectrum of path}
Let $G$ be a directed path of length $\ell$: 
$v_0 \rightarrow v_1 \rightarrow \cdots \rightarrow v_\ell$. 
The associative spectrum of the corresponding graph algebra is
\[
s_n(\GA{G})=
\begin{cases}
\thn{\ell}{n}, & \text{if $n \leq \ell+1$,} \\
\thn{\ell}{n}+1, & \text{if $n \geq \ell+2$.}
\end{cases}
\]
\end{proposition}

\begin{proof}
If $T$ is a DFS tree with $h(T)>\ell$, then there is no homomorphism from $T$ to $G$, 
hence all bracketings corresponding to DFS trees of height at least $\ell+1$ 
belong to the same equivalence class of $\sigma_n(\GA{G})$ by Proposition~\ref{prop:phi}.
(Note that such trees exist only if $n \geq \ell+2$.)
If $h(T)\leq\ell$, then there exist homomorphisms from $T$ to $G$, for instance the
collapsing map $\varphi$ of $T$ on $G$ defined by $\varphi(x_i)=v_{d_T(x_i)}$.
Now if $T'$ is another DFS tree of size $n$, then $\varphi$ is a homomorphism of $T'$ to $G$
if and only if $d_{T'}(x_i)=d_T(x_i)$ for all $i\in\nset{n}$, and this implies $T'=T$ by
Proposition~\ref{prop:unique-depth-seq}. This together with Proposition~\ref{prop:phi}
shows that each bracketing whose DFS tree has height at most $\ell$ forms a singleton class
in $\sigma_n(\GA{G})$. 
There are $\thn{\ell}{n}$ such classes, and if $n \geq \ell+2$, then we
also have the class corresponding to trees of height at least $\ell+1$.
\end{proof}

Next we examine directed paths with some loops. By Theorem~\ref{thm:antiass}
(or just by Lemma~\ref{lem:no-SCC-SCC}), if we have at least two loops on a path,
then the graph is antiassociative, so it suffices to consider the case of only one loop.
We determine the spectrum of directed paths with a loop on the last vertex; the other
cases constitute topic for further research. (However, see Proposition~\ref{prop:spectrum of @-->O}
for the path of length $1$ with a loop on the first vertex.)

\begin{lemma}
\label{lem:level equivalence}
Let $\sim$ be the equivalence relation on $B_n$ that relates 
$t$ and $t'$ if and only if $T:=G(t)$ and $T':=G(t')$ coincide up to 
level $h$, i.e., $\Ltt\geq h$.
Then $\card{B_n/{\sim}}=\thn{h+1}{n}$.
\end{lemma}

\begin{proof}
The equivalence relation $\sim$ on $B_n$ induces naturally an equivalence relation on
the set of zag sequences of length $n$, and we will use the same symbol $\sim$ for this relation.
For any zag sequence $\mathbf{d}=(d_1,\dots,d_n)$, let $\beta(\mathbf{d})=(d'_1,\dots,d'_n)$
be the sequence obtained from $\mathbf{d}$ by replacing each element 
greater than $h$ by $h+1$, i.e., $d'_i=\min(d_i,h+1)$ for $i=1,\dots,n$.%
\footnote{The map $\beta$ can be explained in terms of DFS trees as follows: if $T$ is the DFS tree corresponding to the zag sequence $\mathbf{d}$, then $\beta(\mathbf{d})$ corresponds to the DFS tree obtained from $T$ by turning, for each vertex $v$ at depth $h$, all descendants of $v$ into children of $v$.}
It is straightforward to verify that $\beta(\mathbf{d})$ is also a zag sequence,
and every zag sequence bounded above by $h+1$ is in the image of $\beta$ 
(indeed, if $\mathbf{d}$ is bounded by $h+1$, then $\beta(\mathbf{d})=\mathbf{d}$).
Morover, for all zag sequences $\mathbf{d_1},\mathbf{d_2}$ of size $n$, we have 
$\beta(\mathbf{d_1})=\beta(\mathbf{d_2})$ if and only if $\mathbf{d_1}\sim\mathbf{d_2}$.
Thus $\beta$ is a surjection from the set of all zag sequences of size $n$ to the
set of all zag sequences of size $n$ bounded by $h+1$, and the kernel of $\beta$ is the
equivalence relation $\sim$. This implies that the number of equivalence classes of $\sim$
equals the cardinality of the image of $\beta$, which is clearly $\thn{h+1}{n}$.
\end{proof}

\begin{proposition}
\label{prop:spectrum of path with loop}
Let $G$ be a directed path of length $\ell$ with a loop on the last vertex:
$v_0 \rightarrow v_1 \rightarrow \cdots \rightarrow v_\ell
\ \loopedge$.
The associative spectrum of the corresponding graph algebra is $s_n(\GA{G})=\thn{\ell}{n}$.
\end{proposition}

\begin{proof}
Homomorphisms of a DFS tree $T$ into $G$
are uniquely determined by 
the image of $x_1$: a map $\varphi\colon X_n \to V(G)$ with 
$\varphi(x_1)=v_k$ is a homomorphism from $T$ to $G$ if and only if $\varphi(x_i)=v_{d_T(x_i)+k}$ whenever $d_T(x_i)<\ell-k$
and $\varphi(x_i)=v_\ell$ whenever $d_T(x_i)\geq\ell-k$. 
This implies, by Proposition~\ref{prop:phi}, that $\GA{G}$ satisfies 
a bracketing identity $t \approx t'$ if and only if 
$\Ltt\geq\ell-1$. Therefore, Lemma~\ref{lem:level equivalence} gives 
$s_n(\GA{G})=\thn{\ell-1+1}{n}=\thn{\ell}{n}$.
\end{proof}

For the spectrum of the directed cycle $C_m$, we need to
count depth sequences modulo $m$ by Lemma~\ref{lem:m-cycle}
(or zag sequences modulo $m$, according to 
Proposition~\ref{prop:depth-zag-seq}).
The resulting numbers are called \emph{modular Catalan numbers} in 
\cite{HeinHuang-2017}, and they are denoted by $C_{m,n}$.
For us it will be most convenient to define these numbers simply as
$C_{m,n}:=s_{n+1}(\GA{C_m})$, and we refer the reader to  \cite{HeinHuang-2017}
for plenty of information on these numbers (tables of numerical values, 
references to OEIS entries, formulas and various combinatorial interpretations).
We give two combinatorial interpretations in the next proposition.
The second one is stated in \cite{HeinHuang-2017}, but the proof is left to the
reader there, so we include the proof here.

\begin{proposition}
\label{prop:spectrum of C_m}
The associative spectrum $s_n(\GA{C_m})=C_{m,n-1}$ counts the number of
zag sequences satisfying 
\begin{equation}
\label{eq:zag mod m}
d_{i+1}-d_i \in \{ 2-m,3-m,\dots,0,1 \} \text{ for all } i\in\{1,\dots,n-1\}. 
\end{equation}
Furthermore, $s_n(\GA{C_m})=C_{m,n-1}$ equals the number of 
Dyck paths of semilength $n-1$ that do not contain $D \cdots DU=D^mU$.
\end{proposition}

\begin{proof}
According to Proposition~\ref{prop:depth-zag-seq} and Lemma~\ref{lem:m-cycle},
the associative spectrum of \GA{C_m} counts the number of zag sequences modulo $m$.
We claim that each zag sequence of length $n$ is congruent modulo $m$ 
to exactly one zag sequence $(d_1,\dots,d_n)$ that satisfies  \eqref{eq:zag mod m}.

Since $\{ 2-m,3-m,\dots,0,1 \}$ is a complete system of residues modulo $m$,
it is clear that if two zag sequences satisfying \eqref{eq:zag mod m} are
congruent modulo $m$, then they are equal. (Note that since all zag sequences
start with $0$, the differences $d_{i+1}-d_i$ uniquely determine the zag sequence.)
To prove that every zag sequence is congruent to a zag sequence that
satisfies \eqref{eq:zag mod m}, let $(d_1,\dots,d_n)$ be an arbitrary 
zag sequence, and define the numbers $0=d'_1,\dots,d'_n$ recursively by 
$d'_{i+1}=d'_i+(d_{i+1}-d_i)^\ast$, where $(d_{i+1}-d_i)^\ast$ is
the unique element of the set $\{ 2-m,3-m,\dots,0,1 \}$ that is congruent to 
$d_{i+1}-d_i$ modulo $m$. 
Obviously, we have $d'_{i}\equiv d_i\pmod{m}$ and $d'_{i+1}-d'_i \leq 1$;
we only need to prove that $d'_i\geq 1$ for $i=1,\dots,n$.
Since $d_{i+1}-d_i \leq 1$ by the definition of a zag sequence,
we have $(d_{i+1}-d_i)^\ast \geq d_{i+1}-d_i$, and then an easy induction
argument proves that $d'_i \geq d_i$ for $i=1,\dots,n$. This shows that
$d'_1,\dots,d'_n$ is indeed a zag sequence, hence our claim is proved.

We have proved so far that $s_n(\GA{C_m})$ equals the number of zag sequences 
that satisfy \eqref{eq:zag mod m}.
Note the difference between \eqref{eq:zag} and \eqref{eq:zag mod m}:
an arbitrary zag sequence can have arbitrarily large decreases, while a
sequence satisfying \eqref{eq:zag mod m} can drop 
at most by $m-2$.\footnote{Good news for Sisyphus!}
To prove the statement about Dyck paths, let us rewrite  \eqref{eq:zag mod m}
in terms of the corresponding DFS tree $T$:
\begin{equation}
\label{eq:depth mod m}
d_T(x_{i+1}) \geq d_T(x_i) - (m-2) \text{ for all } i\in\{1,\dots,n-1\}. 
\end{equation}
If $x_{i+1}$ is a child of $x_i$, then this inequality holds trivially
(in this case we have $d_T(x_{i+1})=d_T(x_i)+1$). 
Otherwise, $x_{i+1}$ is a child of one of the ancestors $x_p$ of $x_i$,
thus the depth\hyp{}first search goes down to $x_p$ (which has been visited before)
after the first visit of $x_i$, and then from $x_p$ it takes one step up to reach $x_{i+1}$ for 
the first time. 
This can be seen in the Dyck path (see Remark~\ref{rem:lattice paths}) as a 
sequence of steps $D \cdots DU=D^kU$ from the point labelled by $x_i$ to the 
point labelled by $x_{i+1}$.
(For example, in Figure~\ref{fig:Dyck} we have the steps $DDU$ from the label 
$x_4$ to the label $x_5$.)
The number of down-steps here is $k=d_T(x_i)-d_T(x_p)=d_T(x_i)-d_T(x_{i+1})+1$.
Thus \eqref{eq:depth mod m} is equivalent to
$k \leq m-1$, hence \eqref{eq:zag mod m} means that any sequence of consecutive steps
of the form $D \cdots DU$ in the Dyck path can have at most $m-1$ down-steps.
\end{proof}

\begin{remark}
Proposition~\ref{prop:spectrum of C_m} implies that $s_n(\GA{C_m})$
is nondecreasing in $m$, hence $s_n(\GA{C_m}) \geq s_n(\GA{C_2}) = 2^{n-2}$ for all $m \geq 2$
(see Lemma~\ref{lem:H2}).
The results of \cite{CsaWal-2000} and \cite{HeinHuang-2017} imply that the associative 
spectrum of $\GA{C_m}$ coincides with that of the operation $x+\varepsilon y$ on complex numbers, 
where $\varepsilon$ is a primitive $m$-th root of unity. In particular, for $m=2$, we have that 
the spectrum of subtraction consists of powers of $2$ (see 3.1 in \cite{CsaWal-2000}).
\end{remark}

Now let us study digraphs on two vertices systematically.
Up to isomorphism, there are ten digraphs on two vertices; they are presented in Table~\ref{tab:2-digraphs}.
The corresponding graph algebras are three\hyp{}element groupoids, and the last column of the table indicates the Siena Catalog numbers of their isomorphism class representatives as listed in \cite{BerBur-1996}, as well as the ones of their opposite groupoids.
Of these ten digraphs, only three
are not covered by 
Proposition~\ref{prop:associative} and Theorem~\ref{thm:antiass} (i.e., that are
neither associative nor antiassociative): the undirected path of length one,
the directed path of length one, and the directed path of length one with a loop on the
first vertex. 
The first two ones are special cases of Theorem~\ref{thm:main-undirected}
and Proposition~\ref{prop:spectrum of path}, respectively. 
We treat the third one in Proposition~\ref{prop:spectrum of @-->O}, and for that we need to
investigate an equivalence relation on DFS trees determined by their leaves.

\begin{table}
\newcommand{\PICbox}{\useasboundingbox (-0.7,-0.4) rectangle (1.7,0.4);}
\tikzset{every node/.style={circle,draw,inner sep=1.5,fill=black}, every path/.style={->,>=stealth}}
\begin{tabular}{cccc}
\toprule
$G$ & $s_n(\GA{G})$ & result & Siena Catalog \\
\midrule
\begin{tikzpicture}[scale=1, baseline=(0.south)]
\PICbox
\node (0) at (0,0) {};
\node (1) at (1,0) {};
\end{tikzpicture}
& $1$ & Thm.\ \ref{thm:main-undirected}, Prop.\ \ref{prop:associative} & 1 \\
\begin{tikzpicture}[scale=1, baseline=(0.south)]
\PICbox
\node (0) at (0,0) {};
\node (1) at (1,0) {};
\path (1) edge[out=35,in=325,looseness=18] (1);
\end{tikzpicture}
& $1$ & Thm.\ \ref{thm:main-undirected}, Prop.\ \ref{prop:associative} & 3 \\
\begin{tikzpicture}[scale=1, baseline=(0.south)]
\PICbox
\node (0) at (0,0) {};
\node (1) at (1,0) {};
\path (0) edge[out=215,in=145,looseness=18] (0);
\path (1) edge[out=35,in=325,looseness=18] (1);
\end{tikzpicture}
& $1$ & Thm.\ \ref{thm:main-undirected}, Prop.\ \ref{prop:associative} & 80 \\
\begin{tikzpicture}[scale=1, baseline=(0.south)]
\PICbox
\node (0) at (0,0) {};
\node (1) at (1,0) {};
\path (0) edge (1);
\end{tikzpicture}
& $2$ & Prop.\ \ref{prop:spectrum of path} & 7, 4 \\
\begin{tikzpicture}[scale=1, baseline=(0.south)]
\PICbox
\node (0) at (0,0) {};
\node (1) at (1,0) {};
\path (0) edge (1);
\path (0) edge[out=215,in=145,looseness=18] (0);
\end{tikzpicture}
& $2^{n-2}$ & Prop.\ \ref{prop:spectrum of @-->O} & 9, 55 \\
\begin{tikzpicture}[scale=1, baseline=(0.south)]
\PICbox
\node (0) at (0,0) {};
\node (1) at (1,0) {};
\path (0) edge (1);
\path (1) edge[out=35,in=325,looseness=18] (1);
\end{tikzpicture}
& $1$ & Prop.\ \ref{prop:associative} & 29, 6 \\
\begin{tikzpicture}[scale=1, baseline=(0.south)]
\PICbox
\node (0) at (0,0) {};
\node (1) at (1,0) {};
\path (0) edge (1);
\path (0) edge[out=215,in=145,looseness=18] (0);
\path (1) edge[out=35,in=325,looseness=18] (1);
\end{tikzpicture}
& $C_{n-1}$ & Thm.\ \ref{thm:antiass} & 84, 82 \\
\begin{tikzpicture}[scale=1, baseline=(0.south)]
\PICbox
\node (0) at (0,0) {};
\node (1) at (1,0) {};
\path (0) edge[<->] (1);
\end{tikzpicture}
& $2^{n-2}$ & Thm.\ \ref{thm:main-undirected} & 33, 56 \\
\begin{tikzpicture}[scale=1, baseline=(0.south)]
\PICbox
\node (0) at (0,0) {};
\node (1) at (1,0) {};
\path (0) edge[<->] (1);
\path (1) edge[out=35,in=325,looseness=18] (1);
\end{tikzpicture}
& $C_{n-1}$ & Thm.\ \ref{thm:main-undirected}, Thm.\ \ref{thm:antiass} & 35, 58 \\
\begin{tikzpicture}[scale=1, baseline=(0.south)]
\PICbox
\node (0) at (0,0) {};
\node (1) at (1,0) {};
\path (0) edge[<->] (1);
\path (0) edge[out=215,in=145,looseness=18] (0);
\path (1) edge[out=35,in=325,looseness=18] (1);
\end{tikzpicture}
& $1$ & Thm.\ \ref{thm:main-undirected}, Prop.\ \ref{prop:associative} & 107, 128 \\
\bottomrule 
\end{tabular}

\bigskip
\caption{Digraphs on two vertices and their associative spectra.}
\label{tab:2-digraphs}
\end{table}

Let $T$ and $T'$ be DFS trees on $n$ vertices.
We say that $T$ and $T'$ are \emph{leaf\hyp{}equivalent} if they have the same set of leaves.

\begin{lemma}
\label{lem:leaf-equivalence}
For $n > 1$, the number of leaf\hyp{}equivalence classes of DFS trees on $n$ vertices is $2^{n-2}$.
\end{lemma}

\begin{proof}
The set of leaves of a DFS tree on $n$ vertices is a subset of $X_n$ that does not contain the root $x_1$, but it always contains $x_n$.
On the other hand, it is easy to see that for every subset $S = \{x_{i_1}, x_{i_2}, \dots, x_{i_r}\} \subseteq X_n$ with $1 < i_1 < i_2 < \dots < i_r = n$, there exists a DFS tree whose leaves are precisely the elements of $S$.
For example, we can take the tree
comprising just the paths from $x_1$ to each $x_{i_j} \in S$ that are disjoint except for the initial vertex:
$x_1 \rightarrow x_2 \rightarrow \dots \rightarrow x_{i_1}$
and
$x_1 \rightarrow x_{i_{j-1} + 1} \rightarrow \dots \rightarrow x_{i_j}$ for $2 \leq j \leq r$.
The number of subsets of $X_n$ containing $x_n$ but not containing $x_1$ is $2^{n-2}$.
\end{proof}

\begin{proposition}
\label{prop:spectrum of @-->O}
The associative spectrum $s_n$ of the graph algebra corresponding to the graph $G$ given by 
$V(G)=\{v,w\},~E(G)=\{(v,v),(v,w)\}$ is $s_n=2^{n-2}$.
\end{proposition}
\begin{proof}
For any DFS tree $T$ of size $n$, a map $\varphi\colon X_n\to \{v,w\}$ is a homomorphism 
of $T$ into $G$ if and only if all vertices that are mapped to $w$ are leaves in $T$.
Therefore, $\GA{G}$ satisfies a bracketing identity $t\approx t'$
if and only if the corresponding trees $T:=G(t)$ and $T':=G(t')$ are leaf\hyp{}equivalent.
Now Lemma~\ref{lem:leaf-equivalence} implies that $s_n=2^{n-2}$.
\end{proof}

Finally, we consider some graphs on three vertices.

\begin{proposition}
\label{prop:spectrum of @<--O-->@}
The associative spectrum of the graph algebra corresponding to the graph $G$ given by 
$V(G)=\{u,v,w\},~E(G)=\{(u,v),(v,v),(u,w),(w,w)\}$ is
$s_n(\GA{G})=2^{n-2}$.
\end{proposition}

\begin{proof}
For any DFS tree $T$ of size $n$, a map $\varphi\colon X_n \to \{u,v,w\}$ is a homomorphism of 
$T$ into $G$ if and only if either $\varphi(X_n)=v$ or $\varphi(X_n)=w$, or $\varphi(x_1)=u$ 
and all other vertices are mapped to $\{v,w\}$ in such a way, that if a vertex of depth one 
is mapped to $v$ (to $w$), then all of its descendants are also mapped to $v$ (to $w$):
\[
\forall p \in X_n\colon d_{T}(p)=1 \implies \varphi(V(T_p))=\{\varphi(p)\} \in\{\{v\},\{w\}\}.
\]
Thus the set of all homomorphisms of $T$ into $G$ is determined by the
partition $\{V(T_p) \mid p\in X_n\text{ and } d_{T}(p)=1\}$ 
of the set $\{x_2,\dots,x_n\}$. This partition is in turn determined uniquely 
by the set of depth-one vertices. Indeed, if the depth-one vertices of $T$ are 
$x_{i_1},\dots,x_{i_{s}}$ with $2=i_1<\dots<i_{s}\leq n$, then 
$V(T_{x_{i_k}})=X_{[i_k,i_{k+1}-1]}$ for $k=1,\dots,s-1$
and $V(T_{x_{i_{s}}})=X_{[i_{s},n]}$. By Proposition~\ref{prop:phi}, this implies that 
$\GA{G}$ satisfies a bracketing identity $t\approx t'$ if and only
if $\Ltt \geq 1$. Therefore, by Lemma~\ref{lem:level equivalence} 
we have $s_n(\GA{G})=\thn{2}{n}=2^{n-2}$.
\end{proof}

\begin{proposition}
\label{prop:spectrum of O-->@<-->@}
The associative spectrum of the graph algebra corresponding to the graph $G$ given by 
$V(G)=\{u,v,w\},~E(G)=\{(u,v),(v,v),(v,w),(w,v),(w,w)\}$ is $s_n(\GA{G})=2^{n-2}$.
\end{proposition}

\begin{proof}
For any DFS tree $T$ of size $n$, a map $\varphi\colon X_n\to \{u,v,w\}$ is a homomorphism 
of $T$ into $G$ if and only if either $\varphi(X_n)\subseteq\{v,w\}$, or $\varphi(x_1)=u$, 
all depth-one vertices are mapped to $v$, and the other vertices are mapped to $\{v,w\}$ 
in an arbitrary way. Thus the set of all homomorphisms of $T$ into $G$ is determined uniquely
by the set of depth-one vertices. Therefore, just as in the previous proposition, we can conclude 
$s_n(\GA{G})=\thn{2}{n}=2^{n-2}$ with the help of Lemma~\ref{lem:level equivalence}.
\end{proof}

\begin{remark}
The graph algebra of the directed path of length one with loops on both vertices is isomorphic to the three\hyp{}element groupoid with Siena Catalog number 84 and antiisomorphic to the one with number 82 (see \cite{BerBur-1996}).
These groupoids were shown in \cite[statements 2.4, 5.7]{CsaWal-2000} to be antiassociative;
this result also follows immediately from our Lemma~\ref{lem:no-SCC-SCC}.
\end{remark}

\section*{Acknowledgments}

The authors would like to thank Mikl\'{o}s Mar\'{o}ti and Nikolaas Verhulst for helpful discussions.
This work was partially supported by the Funda\c{c}\~ao para a Ci\^encia e a Tecnologia (Portuguese Foundation for Science and Technology) through the project UID/MAT/00297/2019 (Centro de Matem\'atica e Aplica\c{c}\~oes) and the project PTDC/MAT-PUR/31174/2017.
Research partially supported by the Hungarian Research, Development and Innovation Office grant K115518 and by the Ministry of Human Capacities, Hungary grant 20391-3/2018/FEKUSTRAT.

\end{document}